\documentclass[a4paper,12pt]{article}

\usepackage{pstricks}
\usepackage{graphicx,psfrag}
\usepackage{amsmath}
\usepackage{amssymb}
\usepackage{amsthm}

\usepackage{color}

      \newcommand {\al}   {\alpha}          \newcommand {\bt}  {\beta}
      \newcommand {\gam } {\gamma}          
      \newcommand {\del}  {\delta}          \newcommand {\Del} {\Delta}
              \newcommand {\ve}   {\varepsilon}
                 \newcommand {\vphi} {\varphi}
      \newcommand {\lam}  {\lambda}         
                
      \newcommand {\om}   {\omega}          \newcommand {\Om}  {\Omega}
      \newcommand {\pl}   {\partial}        
           
           \newcommand {\UUU}  {{\cal U}}
            \newcommand {\SSS}  {{\mathcal S}}
       \newcommand {\CCC}  {\mathfrak{C}}     \newcommand {\LLL}  {{\cal L}}
                 \newcommand {\RRR}  {{\mathbb R}}
      \newcommand {\OOO}  {{\cal O}}        \newcommand {\MMM}  {{\cal M}}
           
      \newcommand {\EEE}  {\mathcal{E}}        
      \newcommand {\FFF}   {{\cal F}}

          \newcommand {\AAA}  {{\cal A}}   
          \newcommand {\conv}  {\text{conv}}          \newcommand {\uu}  {u_1}
           \newcommand {\CMu}  {C(u,M)}

     \newcommand {\beq}  {\begin{equation}}
      \newcommand {\eeq}  {\end{equation}}  
     \newcommand {\beqo}  {\begin{equation*}}
      \newcommand {\eeqo}  {\end{equation*}}

      \newtheorem{theorem}{Theorem}
      \newtheorem{lemma}{Lemma}
      \newtheorem{propo}{Proposition}
      \newtheorem{zam}{Remark}
      
      \newtheorem{corollary}{Corollary}

\author{Alexander Plakhov\thanks{Center for R\&{}D in Mathematics and Applications, Department of Mathematics, University of Aveiro, Portugal and Institute for Information Transmission Problems, Moscow, Russia.}}

\title{A solution to Newton's least resistance problem\\ is uniquely defined by its singular set}

\begin{document}

\maketitle

\begin{abstract}
Let $u$ minimize the functional $F(u) = \int_\Omega f(\nabla u(x))\, dx$ in the class of convex functions $u : \Omega \to {\mathbb R}$ satisfying $0 \le u \le M$, where $\Omega \subset {\mathbb R}^2$ is a compact convex domain with nonempty interior and $M > 0$, and $f : {\mathbb R}^2 \to {\mathbb R}$ is a $C^2$ function, with $\{ \xi : \, \text{the smallest eigenvalue of} \, f''(\xi) \, \text{is zero} \}$ being a closed nowhere dense set in ${\mathbb R}^2$. Let epi$(u)$ denote the epigraph of $u$. Then any extremal point $(x, u(x))$ of epi$(u)$ is contained in the closure of the set of singular points of epi$(u)$. As a consequence, an optimal function $u$ is uniquely defined by the set of singular points of epi$(u)$. This result is applicable to the classical Newton's problem, where $F(u) = \int_\Omega (1 + |\nabla u(x)|^2)^{-1}\, dx$.
\end{abstract}

\begin{quote}
\end{quote}

\begin{quote} {\small {\bf Key words and phrases:}
Newton's problem of least resistance, convex geometry, surface area measure, method of nose stretching.}
\end{quote}

\section{Introduction}

{\bf 1.1.}\,
Isaac Newton in his {\it Principia} \cite{N} considered the following problem of optimization. A solid body moves with constant velocity in a sparse medium. Collisions of the medium particles with the body are perfectly elastic. The absolute temperature of the medium is zero, so as the particles are initially at rest. The medium is extremely rare, so that mutual interactions of the particles are neglected. As a result of body-particle collisions, the drag force acting on the body is created. This force is usually called {\it resistance}.

The problem is: given a certain class of bodies, find the body in this class with the smallest resistance. Newton considered the class of convex bodies that are rotationally symmetric with respect to a straight line parallel the direction of motion and have fixed length along this direction and fixed maximal width.

In modern terms the problem can be formulated as follows. Let a reference system $x_1,\, x_2,\, z$ be connected with the body and the $z$-axis coincide with the symmetry axis of the body. We assume that the particles move upward along the $z$-axis. Let the lower part of the body's surface be the graph of a convex radially symmetric function $z = u(x_1, x_2) = \vphi(\sqrt{x_1^2 + x_2^2})$, $x_1^2 + x_2^2 \le L^2$; then the resistance equals
$$
2\pi \rho v^2 \int_0^L \frac{1}{1 + \vphi'(r)^2}\, r\, dr,
$$
where the constants $\rho$ and $v$ mean, respectively, the density of the medium and the scalar velocity of the body. The problem is to minimize 
 the resistance in the class of convex monotone increasing functions $\vphi : [0,\, L] \to \RRR$ satisfying $0 \le \vphi \le M$. Here $M$ and $L$ are the parameters of the problem: $M$ is length of the body and $2L$ is its maximal width.

Newton gave a geometric description of the solution to the problem. The optimal function bounds a convex body that looks like a truncated cone with slightly inflated lateral boundary. An optimal body, corresponding to the case when the length is equal to the maximal width, is shown in Fig.~\ref{figNewton}.
        \begin{figure}[h]
\centering
\hspace*{6mm}
\rotatedown{
\includegraphics[scale=0.45]{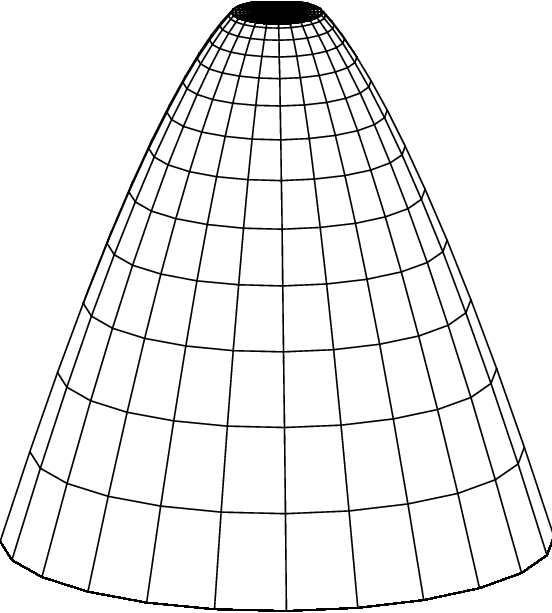}
}
\caption{A solution to the rotationally symmetric Newton problem.}
\label{figNewton}
\end{figure}

Starting from the important paper by Buttazzo and Kawohl \cite{BK}, the problem of minimal resistance has been studied in various classes of (generally) nonsymmetric and/or (generally) nonconvex bodies. The problem for nonconvex bodies is by now well understood \cite{CL1,CL2,Canadian,AP,SIREV,Nonl2016,MMOP}. Generalizations of the problem to the case of rotating bodies have been studied \cite{SIMArough,ARMA,PTG,OMT}, and connections with the phenomena of invisibility, retro-reflection, and Magnus effect in geometric optics and mechanics have been established \cite{invisibility,invisN,camouflage,tube,retro,PTG}. The methods of billiards, Kakeya needle problem, optimal mass transport have been used in these studies. A detailed exposition of results obtained in this area can be found in \cite{bookP}.

The most direct generalization of the original Newton's problem concerns finding the optimal convex (not necessarily symmetric) shape. More precisely, the problem is to minimize
\beq\label{func N}
\int\!\!\!\int_\Om \frac{1}{1 + |\nabla u(x_1,x_2)|^2}\, dx_1 dx_2
\eeq
in the class of convex functions
$$
\CCC_M = \{ u : \Om \to \RRR :\, 0 \le u \le M,\ u  \text{ is convex} \}.
$$
Here $\Om \subset \RRR^2$ is a compact convex set with nonempty interior, and $M > 0$ is the parameter of the problem.

Surprisingly enough, this problem is still poorly understood. It is known that there exists at least one solution \cite{M,BFK}. Let $u$ be a solution; then $u\rfloor_{\pl\Om} = M$ \cite{boundary} and at any regular point $x = (x_1, x_2)$ of $u$ we have either $|\nabla u(x)| \ge 1$, or $|\nabla u(x)| = 0$ \cite{BFK}. Moreover, if the zero level set $L = \{ x : u(x) = 0 \}$ has nonempty interior then we have $\lim_{\stackrel{x \to \bar x}{x \not\in L}} |\nabla u(x)| = 1$ for almost all $\bar x \in \pl L$ \cite{ridge}. If $u$ is $C^2$ in an open set $\UUU \subset \Om$, then the second derivative $u''(x)$ has a zero eigenvalue for all $x \in \UUU$, and therefore, graph$\big( u\rfloor_\UUU \big) = \{ (x, u(x)) : x \in \UUU \}$ is a developable surface \cite{BrFK}.
The problem in special subclasses of convex functions is solved in \cite{LP1},\cite{CaL},\cite{LZ}.
A more detailed review of results concerning the convex problem can be found in \cite{sliding}.

The last property is of special interest. Indeed, the numerical results stated in \cite{LO} and \cite{W} seem to indicate that the singular points of a solution $u$ form several curves on $\Om$ and the graph of $u$ is foliated by line segments and planar pieces of surface outside these curves. That is, cutting the graph of $u$ along the singular curves, one can flatten the resulting pieces of surface on a plane without distortion.
Unfortunately, there are no mathematical results indicating that $u$ is $C^2$ outside the set of singular points. It is even not known whether the domain of $C^2$ smoothness of $u$ is nonempty.

{\bf 1.2.}\,
The aim of the present paper is to relax the $C^2$ condition to the $C^1$ one. This paper is a continuation of \cite{sliding}. The main difference of the results is that in \cite{sliding} the {\it existence} of a solution possessing a certain property is guaranteed, while in this paper it is assured that the property holds for {\it any} solution. It took about 25 years to relax from $C^2$ to $C^1$ but, in our opinion, the method used here and in \cite{sliding} may be more important than the result obtained. The method is called {\it nose stretching} and amounts to a small variation of a convex set
in the following way. Having a convex set $C$ in $\RRR^3$, we take a point $O$ or two points $A$ and $B$ outside $C$ near its boundary, define $\tilde C = \text{{\conv}}(C \cup \{ O \})$ (in \cite{sliding}) or $\tilde C = \text{{\conv}}(C \cup \{ A, B \})$ (in this paper), and then define a 1-parameter family of convex sets containing $C$ and $\tilde C$.

We believe that it is fruitful to study the minimization problem in a more general form:

\begin{quote}
Minimize the functional
\beq\label{func general}
F(u) = \int_\Om f(\nabla u(x))\, dx
\eeq
in the class $\CCC_M$, where $f : \RRR^2 \to \RRR$ is a continuous function.\footnote{Since the set of regular points of $u$ is a full-measure subset of $\Om$ and the vector function $x \mapsto \nabla u(x)$ is measurable, $F(u)$ is well defined.}
\end{quote}

Taking $f(\xi) = 1/(1 + |\xi|^2)$, one obtains the functional \eqref{func N}.

Let us also mention an equivalent setting of the problem, first proposed in \cite{BG}. The graph of $u$ is the lower part of the boundary of the convex body
\begin{equation*}\label{C_u}
\CMu = \{ (x,z) :\, x \in \Om,\ u(x) \le z \le M \},
\end{equation*}
and the functional in \eqref{func general} can be represented in the form $F(u) = \FFF(\CMu)$, where
\beq\label{func body}
\FFF(C) = \int_{\pl C} g(n_\xi)\, d\xi.
\eeq
Here $n_\xi$ is the outward normal to the convex body $C$ at a regular point $\xi \in \pl C$, and for $n = (n_1, n_2, n_3) \in S^2$,
$$
g(n) = \left\{ \begin{array}{ll}
-n_3 f\big( \frac{n_1}{n_3},\, \frac{n_2}{n_3} \big), & \text{if} \ \, n_3 < 0; \\
0 & \text{if} \ \, n_3 \ge 0
\end{array} \right.
$$
(see \cite{sliding}). In particular, if $f(\xi) = 1/(1 + |\xi|^2)$, we have $g(n) = \big( n_3 \big)_-^3$, where $( \cdot )_-$ means the negative part of a real number, $z_- = \max \{ -z, 0 \}$.

Correspondingly, the minimization problem for $F$ in \eqref{func general} can be stated in the equivalent form:

\begin{quote}
Minimize $\FFF(C)$ in \eqref{func body} in the class of convex bodies $\{ C :\, C_1 \subset C \subset C_2 \}$, where $C_2 = \Om \times [0,\, M]$ is the cylinder with the base $\Om$ and height $M$ and $C_1 = \Om \times \{ M \}$ is its upper end.
\end{quote}

\begin{zam}
This problem admits a natural mechanical interpretation. Imagine a body moving in a highly rarefied medium where Newton's assumptions are generally not satisfied: the absolute temperature is nonzero and/or the body-particle reflections are not elastic. In this case the resistance (the projection of the drag force on the direction of motion) is given by the functional \eqref{func general} (or, equivalently, by \eqref{func body}), where the function $f$ (or $g$) can be determined if we know the law of body-particle reflection, the temperature, and composition of the medium.
\end{zam}

The $C^2$ result in the paper \cite{BrFK} was formulated and proved for Newton's case $f(\xi) = 1/(1 + |\xi|^2)$. Its natural generalization provided in \cite{sliding} reads as follows.

\begin{theorem}\label{tpropo1} (Theorem 1 in \cite{sliding}).
Let $f$ be a $C^2$ function and the second derivative $f''(\xi)$ have at least one negative eigenvalue for all values of the argument $\xi$. Assume that $u$ minimizes functional \eqref{func general} in the class $\CCC_M$ and $u$ is $C^2$ in $\UUU$, where $\UUU \subset \Om$ is an open set. Then $\det u''(x) = 0$ for all $x \in \UUU$.
\end{theorem}

The $C^1$ result in \cite{sliding} is as follows.

\begin{theorem}\label{tpropo2} (Corollary 3 in \cite{sliding}).
Let $f$ be a bounded continuous function. Then there exists a function $u$ minimizing functional \eqref{func general} in $\CCC_M$ and possessing the following property. If $u$ is $C^1$ in $\UUU$ and $u > 0$ in $\UUU$, where $\UUU \subset \Om$ is an open set, then {\rm graph}$\big( u\rfloor_\UUU \big)$ does not contain extreme points of the epigraph of $u$.
\end{theorem}

\begin{zam}
The words {\rm "there exists a solution"} in Theorem \ref{tpropo2} cannot be replaced with {\rm "for any solution"}. Actually, if the function $f$ is locally affine, there may exist a solution $u$ and an open set $\UUU$ such that $u$ is $C^1$ and positive in $\UUU$, and {\rm graph}$\big( u\rfloor_\UUU \big)$ contains extreme points. Consider two examples.

1. $f$ is piecewise affine, $f(\xi) = (-\langle a, \xi \rangle + b)_+$ for a certain vector $a = (a_1, a_2)$ and for $b > 0$. Here and in what follows, $\langle \cdot \,, \cdot \rangle$ means scalar product. This function corresponds to the mechanical model when the particles of the incident flow with velocities equal to $(a_1, a_2, b)$ get stuck on the body's surface after the collision. In this case any function $u \in \CCC_M$ satisfying $\langle a, \nabla u \rangle < b$ is a solution.

2. $f \ge 0$, and $f(\xi) = 0$ when $|\xi| \le r$. Here $r > 0$.Then any function $u \in \CCC_M$ with $|\nabla u| \le r$ is a solution.

Note that in both examples the set $\{ \xi : \det f''(\xi) = 0 \} \subset \RRR^2$ is large: it is the complement of the line $\langle a, \xi \rangle = b$ in the first example, and contains the open ball $\{ |\xi| < r \}$ in the second one.
\end{zam}

Given a convex set $C \subset \RRR^3$, a point $\xi \in \pl C$ is called {\it regular}, if there is a single plane of support to $C$ at $\xi$, and {\it singular} otherwise. The set of singular points of $\pl C$ is denoted as sing$(C)$. The set of extremal points of $C$ is denoted as ext$(C)$. Let {\rm epi}$(u) = \{ (x,z) : x \in \Om,\, z \ge u(x) \} \subset \RRR^3$ denote the epigraph of $u$. A point $x$ in the interior of $\Om$ is called {\it regular point} of $u$, if there exists $\nabla u(x)$, and {\it singular} otherwise. The set of singular points of $u$ is denoted as sing$(u)$. We have
$$
\text{\rm graph}\big( u\rfloor_{\text{\rm sing}(u)} \big) \subset \text{\rm sing}\, (\text{\rm epi}(u)) \subset \text{\rm graph}\big( u\rfloor_{\text{\rm sing}(u)} \big) \cup (\pl\Om \times \RRR).
$$

Later on in the text we will use the following notation. Let $G$ be a Borel subset of graph$(u)$ and pr$_x(G)$ be its orthogonal projection on the $x$-plane; then by definition
$$
F(G) = \int_{\text{pr}_x(G)} f(\nabla u(x)) dx.
$$
It is easy to check that $F(G)$ is well defined; that is, if $G \subset \text{graph}(u_1)$ and $G \subset \text{graph}(u_2)$ for two convex functions $u_1$ and $u_2$, then $\int_{\text{pr}_x(G)} f(\nabla u_1(x)) dx = \int_{\text{pr}_x(G)} f(\nabla u_2(x)) dx$. Additionally, if $G_2$ is homothetic to $G_1$ with ratio $r \ge 0$, that is, $G_2 = rG_1 + v$ with $v \in \RRR^3$, then $F(G_2) = r^2 F(G_1).$

The main result of this paper is the following theorem.

\begin{theorem}\label{cor}
Let $f$ be a $C^2$ function, and let $\{ \xi : \, \text{the smallest eigenvalue of} \, f''(\xi) \, \text{is zero} \}$ be a closed nowhere dense set in $\RRR^2$. If $u$ minimizes functional \eqref{func general} in $\CCC_M$ then
$$
\text{\rm ext} (\text{\rm epi}(u)) \subset \overline{\text{\rm sing} (\text{\rm epi}(u))}.
$$
Here and in what follows, bar means closure.
\end{theorem}

It follows from this theorem that an optimal function $u$ is uniquely defined by the set of singular points of epi$(u)$.

Recall that $C(u) = \CMu = \{ (x,z) : x \in \Om, \ u(x) \le z \le M \} = \text{epi}(u) \cap \{ z \le M \}$.

The following statement is a corollary of Theorem \ref{cor}.

\begin{corollary}\label{cor1}
Under the conditions of Theorem \ref{cor}, if $u$ minimizes functional \eqref{func general} in the class $\CCC_M$, then $C(u)$ is the closure of the convex hull of {\rm sing}$(C(u))$,
$$
C(u) = \overline{\text{\rm \conv}(\text{\rm sing}\,(C(u)))}.
$$
\end{corollary}

\begin{proof}
We have
$$
\big[ C(u) \cap \text{\rm sing} (\text{\rm epi}(u)) \big] \cup (\pl\Om \times \{ M \}) \subset \text{\rm sing}\,(C(u))
$$ $$
\text{and} \quad \text{\rm ext}\,(C(u)) \subset \big[ C(u) \cap \text{\rm ext} (\text{\rm epi}(u)) \big]  \cup (\pl\Om \times \{ M \}),
$$
hence by Theorem \ref{cor},
$$
\text{\rm ext}\,(C(u)) \subset \big[ C(u) \cap \text{\rm ext} (\text{\rm epi}(u)) \big] \cup (\pl\Om \times \{ M \})
$$ $$
\subset \big[ C(u) \cap \overline{\text{\rm sing} (\text{\rm epi}(u))} \big] \cup (\pl\Om \times \{ M \}) \subset \overline{\text{\rm sing}\,(C(u))},
$$
and by Minkowski's theorem,
$$
C(u) = \conv\big( \text{\rm ext}\,(C(u)) \big) \subset \conv\big( \overline{\text{\rm sing}\,(C(u))} \big) = \overline{\text{\rm \conv}(\text{\rm sing}\,(C(u)))}.
$$
The inverse inclusion is obvious.
\end{proof}

Theorem \ref{cor} and Corollary \ref{cor1} are applicable to the classical case $f(\xi) = 1/(1 + |\xi|^2)$, since $f$ is $C^2$ and the smallest eigenvalue of $f''$ is always negative, and therefore, the set $\{ \xi : \, \text{the smallest eigenvalue of} \, f''(\xi) \, \text{is zero} \}$ is empty.

We will derive Theorem \ref{cor} from the following Theorem \ref{t3} and Lemma \ref{lt1}.

\begin{theorem}\label{t3}
Let $f : \RRR^2 \to \RRR$ be a $C^2$ function.
Assume that a convex function $u : \Om \to \RRR$, an open set $\UUU \subset \Om$, and a point $\check x \in \UUU$ are such that

(i) $u$ is $C^1$ in $\UUU$;

(ii) $(\check x, u(\check x))$ is an extreme point of epi$(u)$;

(iii) the smallest eigenvalue of $f''(\nabla u(\check x))$ is nonzero. \\
Then for any $\ve > 0$ there is a convex function $\uu$ on $\Om$ such that\,

{\rm (a)} $\uu\rfloor_{\Om\setminus\UUU} = u\rfloor_{\Om\setminus\UUU}$,

 {\rm (b)}~$|u - \uu| < \ve$, and

 {\rm (c)}~$F(\uu) < F(u)$.
\end{theorem}

\begin{lemma}\label{lt1}
Consider a convex set $C \subset \RRR^3$ with nonempty interior.
Take an open (in the relative topology) set $U \subset \pl C$ and suppose that all points of $U$ are regular. Take a set ${\cal E} \subset S^2$ that contains no open (in the relative topology) sets, and suppose that each point $r \in U$ with $n_r \not\in \EEE$ is not an extreme point of $C$. Then ${U}$ does not contain extreme points of $C$.
\end{lemma}

The proofs of Theorem \ref{t3} and Lemma \ref{lt1} are given in Section \ref{sec theorem} and Section \ref{sec technical}, respectively.

{\it Proof of Theorem \ref{cor}}.
Denote $C = \text{epi}(u)$,\, $U = \pl C \setminus \overline{\text{\rm sing} (C)}$, and $\EEE = \EEE_1 \cup \EEE_2 \cup \EEE_3  \subset S^2$, where $\EEE_1 = \{ (0,0,-1) \}$,\, $\EEE_2 = \{ (\xi,0) :\, |\xi| = 1 \}$, with $\xi = (\xi_1, \xi_2) \in \RRR^2$, and
$$
\EEE_3 = \Big\{ \frac{(\xi, -1)}{\sqrt{|\xi|^2 + 1}} : \, \text{the smallest eigenvalue of} \, f''(\xi) \, \text{is zero} \Big\}.
$$
Since by the hypothesis of Theorem \ref{cor}, the set $\{ \xi : \, \text{the smallest eigenvalue of} \, f''(\xi) \, \text{is zero} \}$ is nowhere dense, and therefore, contains no open sets, we conclude that the sets $\EEE_3$, and therefore $\EEE$, contain no open sets.

Suppose that a point $\check r = (\check x, \check z) \in U$, with $n_{\check{r}} \not\in \EEE$, is an extreme point of $C$. Since $n_{\check r} \not\in \EEE_2$, we have $\check x \not\in \pl\Om$, and therefore, $\check r$ lies on the graph of $u$, that is, $\check z = u(\check x)$. Since $\check x$ lies in the interior of $\Om$, we have $\check z < M$. Since $n_{\check r} \not\in \EEE_1$, we conclude that $\check z > 0$.

Take a value $\ve > 0$ and an open set $\UUU' \subset \Om$ containing $\check x$ such that graph$( u\rfloor_{\UUU'} )$ is contained in $U$ and $\ve < u < M - \ve$ in $\UUU'$. Since $n_{\check r} \not\in \EEE_3$, the smallest eigenvalue of $f''(\nabla u(\check x))$ is nonzero.
By Theorem \ref{t3}, there exists a convex function $\uu$ on $\Om$ that coincides with $u$ outside $\UUU'$ and satisfies $u-\ve < \uu < u + \ve$ in $\UUU'$, and therefore, belongs to $\CCC_M$, and such that $F(\uu) < F(u)$. We have a contradiction with optimality of $u$.

This contradiction implies that each point $r \in U$ with $n_{r} \not\in \EEE$ is not an extreme point of $C$. Applying Lemma \ref{lt1}, one concludes that $U$ does not contain extreme points of $C$. It follows that $\text{\rm ext} (C) \subset \overline{\text{\rm sing} (C)}$.
 Theorem \ref{cor} is proved.
\hfill $\Box$


\begin{zam}
Theorem \ref{cor} implies that the graph of a solution $u$ minus the closure of the set of its singular points is a developable surface.
Still, nothing is known about the set of singular points. We cannot even guarantee that it is not dense in the graph.
\end{zam}

We state the following
\vspace{2mm}

{\bf Conjecture.} {\it Let $u$ solve problem \eqref{func general} in $\CCC_M$ with $M > 0$. Then the set of singular points of $u$ is a closed nowhere dense subset of $\Om$.}

\section{Proof of Theorem \ref{t3}}\label{sec theorem}

\subsection{A toy example}
The proof contains many technical details, but its main idea is quite simple. To make it clearer, let us first consider the following toy problem in the 2D case.

\begin{propo}\label{light}
Consider the functional $F(u) = \int_a^b f(u'(x)) dx$ defined in the class of convex functions $u : [a,\, b] \to \RRR$, where $f : \RRR \to \RRR$ is a $C^2$ function. Assume that a convex function $u$ and a point $\check x \in (a,\, b)$ are such that

(i) $u$ is $C^1$ in a neighborhood of $\check x$;

(ii) $(\check x, u(\check x))$ is an extreme point of epi$(u)$;

(iii) $f''(u'(\check x)) \ne 0$. \\
Then for any $\ve > 0$ there is a convex function $\uu$ on $\Om$ such that\,

{\rm (a)} $\uu = u$ outside the $\ve$-neighborhood of $\check x$,

 {\rm (b)}~$|u - \uu| < \ve$, and

 {\rm (c)}~$F(\uu) < F(u)$.
\end{propo}

The proof given below is not the easiest one, but the underlying idea of small variation admits a generalization to the 3D case.
However, the 3D case is much more complicated, as will be seen later.

\begin{proof}
Without loss of generality one can assume that the following condition is satisfied:

(iv) $u'(\check x) > u(x)$ for all $x < \check x$ and $u'(\check x) < u(x)$ for all $x > \check x$;\\
otherwise just slightly vary the point $\check x$.

We use the notation $C = \text{epi}(u)$.
Take a point $O = (\check x, \check z)$ below $C$ (so as $\check z < u(\check x)$), and consider two lines of support through $O$ to $C$. In general, the intersection of each line with $C$ is a segment, but slightly moving $O$ in the vertical direction, one can ensure that both these segments are  points. Let them be denoted as $A_0 = (x_1, u(x_1))$ and $B_0 = (x_2, u(x_2))$. We also denote $\bar A = (a, u(a))$ and $\bar B = (b, u(b))$; see Fig.~\ref{fig 2D}\,(a),\,(b). Using condition (iv), one can choose $O$ sufficiently close to $C$ so as $\langle 1 \rangle$ all points of $[x_1,\, x_2]$ are regular points of $u$, $\langle 2 \rangle$ $f''(u'(x))$ does not change the sign in $[x_1,\, x_2]$, and $\langle 3 \rangle$, the values $\check x - x_1$ and $x_2 - \check x$ are smaller than $\ve$.

Denote $\tilde C = \conv(C \cup O)$ and define the family of convex sets
$$
C_s =
\left\{
\begin{array}{ll}
(1-s)C + s\tilde C, & \text{if } 0 \le s \le 1;\\
C \cap \big[ (1-s)C + sO \big], & \text{if } s < 0.
\end{array}\right.
$$
Let $u^{(s)}$ be the function such that epi$(u^{(s)}) = C_s$. In particular, one has $u^{(0)} = u$. Of course, all functions $u^{(s)}$, $s \le 1$ are convex and are defined on $[a,\, b]$. See Fig.~\ref{fig 2D}.
          \begin{figure}[h]
\centering
\includegraphics[scale=0.3]{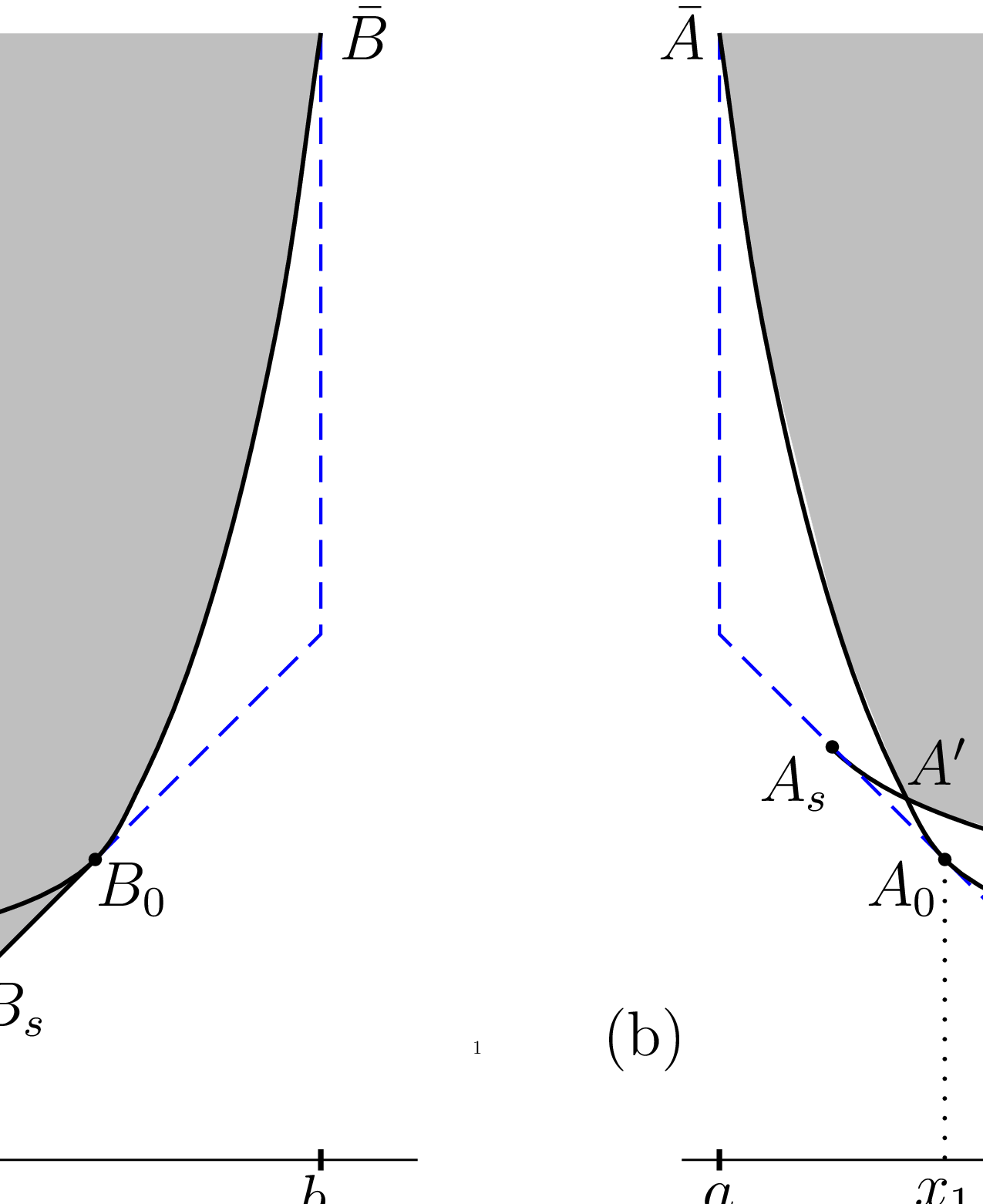}
\caption{The set $C_s$ is shown light gray (a) for $0 < s < 1$ and (b) for $s < 0$.}
\label{fig 2D}
\end{figure}

Denote by $F(\gam)$ the resistance of the curve $\gam$; in particular, the resistance of the arc $A_0 B_0$ is denoted as $F(A_0 B_0)$, and the resistances of the line segments $OA_0$ and $OB_0$ are $F(OA_0)$ and $(OB_0)$. The following relations hold,
\beq\label{f''positive}
\text{if} \ \, f''(u'(x)) > 0 \quad \text{in}\ \ [x_1,\, x_2], \ \ \text{then} \ \, F(OA_0) + F(OB_0) - F(A_0 B_0) > 0;
\eeq
\beq\label{f''negative}
\text{if} \ \, f''(u'(x)) < 0 \quad \text{in}\ \ [x_1,\, x_2], \ \ \text{then} \ \, F(OA_0) + F(OB_0) - F(A_0 B_0) < 0.
\eeq
Indeed, $f$ can be represented as the sum of an affine function and a function that equals zero when the argument equals $u'(x_1)$ or $u'(x_2)$. One easily checks that if $f$ is affine then $F(OA_0) + F(OB_0) = F(A_0 B_0)$. Therefore it suffices to prove \eqref{f''positive} and \eqref{f''negative} for a function $f$ satisfying $f(u'(x_1)) = f(u'(x_2)) = 0$. In this case we have $ F(OA_0) = F(OB_0) = 0$, and if $f'' > 0$ then $f(u'(x)) < 0$ in $[x_1,\, x_2]$, and therefore, $F(A_0 B_0) = \int_{x_1}^{x_2} f(u'(x)) dx < 0$, and \eqref{f''positive} is proved. Similarly, if $f'' < 0$ then $F(A_0 B_0) > 0$, and \eqref{f''negative} is proved.

The graph of $u$ is the union of curves $\bar A A_0$,\, $A_0 B_0$, and $B_0 \bar B$, and therefore,
$$
F(u) = F(\bar A A_0) + F(A_0 B_0) + F(B_0 \bar B).
$$

For $0 \le s \le 1$ the graph of $u^{(s)}$ is the union of curves $\bar A A_0$ and $B_0 \bar B$, line segments $A_0 A_s$ and $B_0 B_s$, and the curve $A_s B_s$; see Fig.~\ref{fig 2D}\,(a). The arc $A_s B_s$ is homothetic to $A_0 B_0$ with the ratio $1-s$ and the center at $O$, therefore $F(A_s B_s) = (1-s) F(A_0 B_0)$. The length of the segment $A_0 A_s$ is $s$ times the length of of the segment $OA_0$, hence $F(A_0 A_s) = sF(OA_0)$. Similarly, $F(B_0 B_s) = sF(OB_0)$. Thus,
$$
F(u^{(s)}) = F(\bar A A_0) + F(A_0 A_s) + F(A_s B_s) + F(B_s B_0) + F(B_0 \bar B)
$$ $$
= F(\bar A A_0) + F(\bar B B_0) + (1-s) F(A_0 B_0) + sF(OA_0) + sF(OB_0).
$$
It follows that for $0 \le s \le 1$, $F(u^{(s)})$ is linear in $s$.

For $s < 0$ the graph of $u^{(s)}$ is the union of curves $\bar A A'$,\, $A' B'$, and $B' \bar B$; see Fig.~\ref{fig 2D}\,(b). Correspondingly, the resistance is
$$
F(u^{(s)}) = F(\bar A A') + F(A' B') + F(B' \bar B)
$$ $$
= F(\bar A A_0) + F(\bar B B_0) + F(A_s B_s) - F(A_0 A_s) - F(B_0 B_s)
$$ $$
+ [F(A_0 A_s) - F(A_0 A') - F(A' A_s)] + [F(B_0 B_s) - F(B_0 B') - F(B' B_s)]
$$

$F(A_0 A_s)$ equals $f(u'(x_1))$ times the length of the projection of $A_0 A_s$ on the $x$-axis, which is proportional to $s$, and $F(A_0 A') + F(A' A_s)$ equals the integral of $f(u'(x_1) + o(1))$ over this projection. It follows that $F(A_0 A_s) - F(A_0 A') - F(A' A_s)$ is $o(s)$ as $s \to 0^-$, and the same is true for $F(B_0 B_s) - F(B_0 B') - F(B' B_s)$. Hence we obtain
$$
F(u^{(s)}) = F(\bar A A_0) + F(\bar B B_0) + (1-s) F(A_0 B_0) + sF(OA_0) + sF(OB_0) + o(s) \ \ \text{as} \ \, s \to 0^-.
$$
It follows that there exists the derivative of $F(u^{(s)})$ at $s = 0$,
$$
\frac{d}{ds}\bigg\rfloor_{s=0} F(u^{(s)}) = F(OA_0) + F(OB_0) - F(A_0 B_0).
$$
According to formulas \eqref{f''positive} and \eqref{f''negative}, the derivative is nonzero, and therefore, choosing $\uu = u^{(s)}$ for $s$ in a sufficiently small one-sided (left or right) neighborhood of 0, one obtains $F(\uu) < F(u)$ .
\end{proof}

Let us now proceed to the proof of Theorem \ref{t3}. We consider separately, in subsections 2.2 and 2.3, the cases when the smallest eigenvalue of $f''(\nabla u(\check x))$ is positive and when it is negative.

\subsection{The smallest eigenvalue of $f''(\nabla u(\check x))$ is positive}
In this case $f''(\nabla u(x))$ is positive definite for $x$ in a neighborhood of $\check x$. The argument closely follows the proof of Lemma 1 in \cite{LP3}.

Without loss of generality one can assume that $f''(\nabla u(x))$ is positive definite for $x$ in $\UUU$. Denote $D_r = \{ (x,z) :\, |x - \check x| < r,\, |z - u(\check x)| < r \}$ with $r = \min \{ \ve, \text{dist}(\check x, \pl\UUU) \}$. Since the point $(\check x, u(\check x))$ is extreme, it is not contained in $\text{{\conv}}(\text{epi}(u) \setminus D_r)$. Draw a plane $z = l(x)$ strictly separating $(\check x, u(\check x))$ and $\text{{\conv}}(\text{epi}(u) \setminus D_r)$. Of course, $l(\check x) > u(\check x)$, and the function $l$ is affine, so as $\nabla l$ is constant. Let $\uu = \max \{ u,\, l \}$.

By construction, the function $\uu$ satisfies conditions (a) and (b) of Theorem \ref{t3}. It remains to check condition (c): $F(\uu) < F(u)$.

The set $\UUU_1 = \{ x : u(x) \ne l(x) \}$ is open and belongs to $\UUU$, and
$$F(u) -F(\uu) = \int_{\UUU_1} \big[ f(\nabla u(x)) - f(\nabla l) \big] dx.$$
Note that since $u = l$ on $\pl\UUU_1$, we have $\int_{\UUU_1}  (\nabla u(x) - \nabla l) dx = 0.$
For $x \in \UUU_1$ and $\nabla u(x) \ne \nabla l$,
$$
f(\nabla u(x)) > f(\nabla l) + \langle f'(\nabla l),\, \nabla u(x) - \nabla l \rangle,
$$
hence
$$
0 < \int_{\UUU_1} \big[ f(\nabla u(x)) - f(\nabla l) \big] dx - \left\langle f'(\nabla l),\
\int_{\UUU_1}  (\nabla u(x) - \nabla l) dx
\right\rangle.
$$ $$
= \int_{\UUU_1} \big[ f(\nabla u(x)) - f(\nabla l) \big] dx = F(u) -F(\uu).
$$
Thus, condition (c) is also proved.


\subsection{The smallest eigenvalue of $f''(\nabla u(\check x))$ is negative}\label{sub neg}
This case is much more difficult than the previous one.

Later on in the proof we use the notation $C = \text{epi}(u)$.

Choose orthogonal coordinates $x_1,\, x_2$ in such a way that $(1,0)$ is an eigenvector of $f''(\nabla u(\check x))$ corresponding to a negative eigenvalue. We have
$$
\frac{d^2}{dt^2}\Big\rfloor_{t=0} f(\nabla u(\check x) + (t,0)) < 0.
$$
Since $f$ is $C^2$ and $\nabla u$ is continuous in $\UUU \ni \check x$, the function $\frac{d^2}{dt^2} f(\nabla u(x) + (t,0))$ is continuous in $x$ and $t$ for $x$ sufficiently close to $\check x$ and $t$ sufficiently close to 0, hence this function is negative for $|x - \check x| < \om,\, |t| < \om$, with $\om > 0$ sufficiently small. Therefore for $|x - \check x| < \om$ the function of one variable $t \mapsto f(\nabla u(x) + (t,0))$, $t \in (-\om,\, \om)$ is strictly concave.

Choose $r$ sufficiently small, so as

$$ r < \text{dist}(\check x, \pl\UUU), \quad \text{and}
$$
   \beq\label{(beta)}
   \begin{array}{l}
   \text{if $|x - \check x| < r,\, |x' - \check x| < r$,  and $\nabla u(x) - \nabla u(x')$ is parallel to the $x_1$-axis},\\
    \text{then the restriction } \text{of $f$ on the line segment $[\nabla u(x),\, \nabla u(x')]$}\\
    \text{is a strictly concave function of one variable.}
   \end{array}
\eeq

Denote $D_r = \{ (x,z) :\, |x - \check x| < r,\, |z - u(\check x)| < r \}$. Since the point $(\check x, u(\check x))$ is extreme, it is not contained in $\text{{\conv}}(C \setminus D_r)$. Draw a plane $\Pi$ strictly separating $(\check x, u(\check x))$ and $\text{{\conv}}(C \setminus D_r)$; that is, the point $(\check x, u(\check x))$ and the convex body $\text{{\conv}}(C \setminus D_r)$ are contained in different open subspaces bounded by $\Pi$ (see Fig.~\ref{fig extreme}).

The plane $\Pi$ divides the surface $\pl C$ into two parts: the upper one, $\pl_+ C$, containing $\pl C \setminus D_r$, and the lower one, $\pl_- C$, containing $(\check x, u(\check x))$. Draw all planes of support through the points of $\pl_+ C$ and denote by $C'$ the intersection of half-spaces that are bounded by these planes and contain $C$. It is seen from the construction that $\pl C \setminus \pl C'$ belongs to the graph of $u\rfloor_\UUU$, and therefore, contains only regular points of $\pl C$.

The body $C$ is contained in $C'$, but does not coincide with it. Indeed, let $\Pi'$ be the plane of support to $C'$ parallel to $\Pi$ and below $\Pi$.
Each point $Q$ in $\Pi' \cap C'$ is a singular point of $\pl C'$, since there are at least two planes of support through $Q$: $\Pi'$ and a plane tangent to $C$ at a point of $\pl C \cap \Pi$. Therefore, $Q$ does not belong to $C$.

Take a point $(x^0, u(x^0))$ in $\pl C \setminus \pl C'$. Let us show that without loss of generality one can assume that $u_{x_2}(x^0) = 0$. Indeed,
               \begin{figure}[h]
\centering
\includegraphics[scale=0.3]{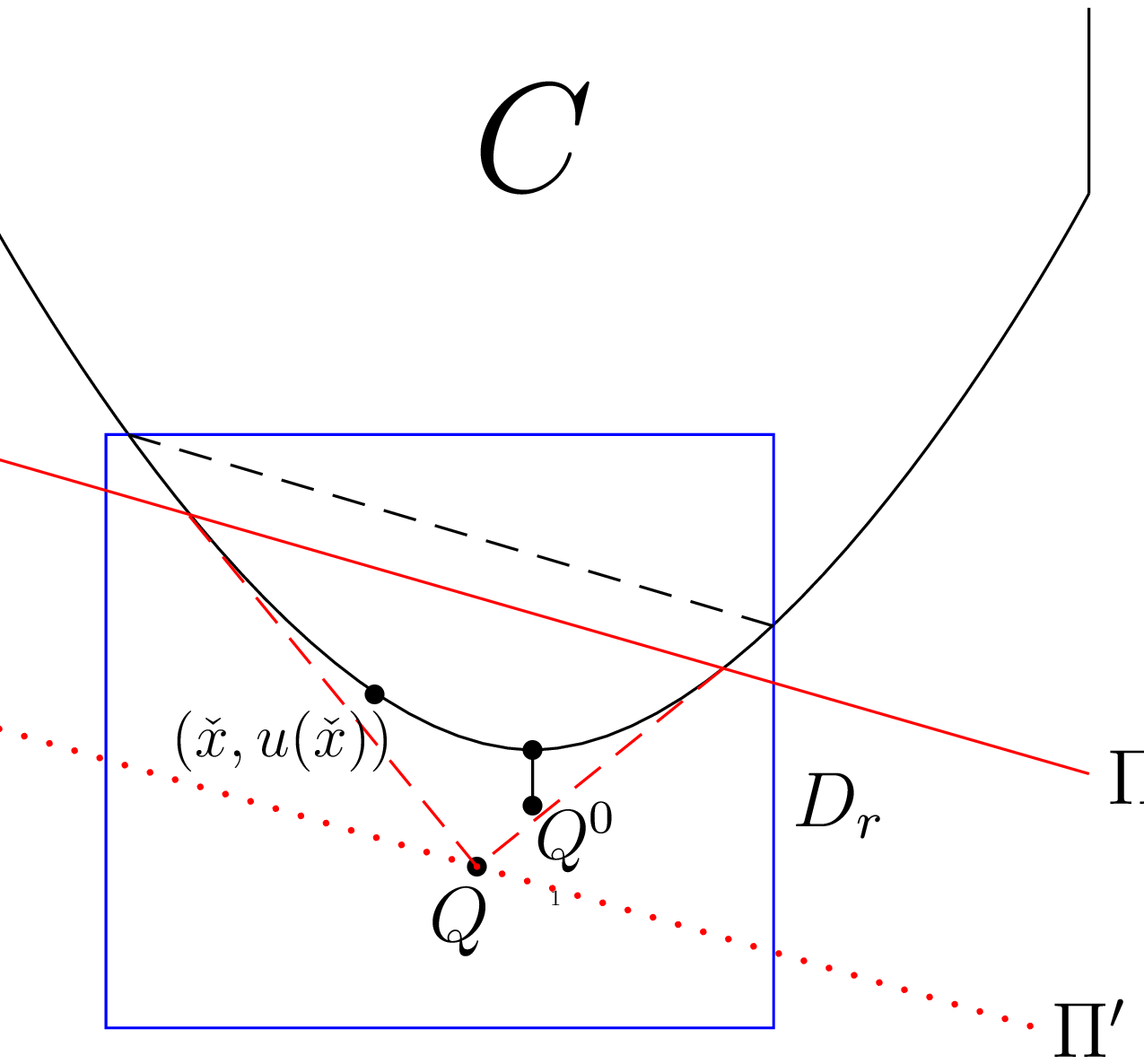}
\caption{All points of $\pl C$ situated below the plane $\Pi$ are regular. The dashed broken line through $Q$ indicates a part of $\pl C'$.}
\label{fig extreme}
\end{figure}
denote $b = u_{x_2}(x^0)$ and consider the functions $\mathbf{u}(x) = u(x) - bx_2$ and $\mathbf{f}(\xi) = f(\xi + (0,b))$. We have $\mathbf{u}_{x_2}(x^0) = 0$. One easily sees that the functions $\mathbf{u}$, $\mathbf{f}$, the set $\UUU$, and the point $\check x$ satisfy the hypotheses (i), (ii), and (iii) of Theorem \ref{t3}. (Note in passing that graph$(\mathbf{u})$ and epi$(\mathbf{u})$ are the images of graph$(u)$ and $C$, respectively, under the map $(x_1, x_2, z) \mapsto (x_1, x_2, z - bx_2)$.)


It suffices to prove the statement of Theorem \ref{t3} for $\mathbf{u}$ and $\mathbf{f}$; that is, fix $\ve > 0$ and find a function $\mathbf{\uu}$ satisfying conditions (a), (b), (c) indicated there. Then, taking $\uu(x) = \mathbf{\uu}(x) + bx_2$ and using that $\int_\Om \mathbf{f}(\nabla \mathbf{\uu}(x)) dx = \int_\Om f(\nabla\uu(x)) dx$ and $\int_\Om \mathbf{f}(\nabla {\mathbf{u}}(x)) dx = \int_\Om f(\nabla u(x)) dx$, one concludes that the statement of Theorem \ref{t3} is also true for the original functions $u$ and $f$.

In what follows we assume that $u_{x_2}(x^0) = 0$.
Consider the auxiliary function of one variable
$$
w(x_1) = \inf_{x_2} u(x_1,x_2).
$$
The function $w$ is convex, $w(x^0_1) = u(x^0)$, and the epigraph of $w$ coincides with the image of $C$ under the projection $\pi: (x_1, x_2, z) \mapsto (x_1, z)$, that is, epi$(w) = \pi(C)$.

Take a vertical interval $J$ outside $C$ with the endpoints $(x^0, u(x^0))$ and $Q^0 = (x^0, z^0)$,\ $z^0 < u(x^0)$, so as $Q^0$ lies in the interior of $C'$.
Draw two support lines through an arbitrary point of the interval $\pi(J) = \big( (x_1^0, z^0),\ (x_1^0, u(x^0)) \big)$ to the convex set $\pi(C)$ in the $(x_1,z)$-plane. Since $\pl(\pi(C))$ contains at most countably many line segments, we conclude that for all points of $\pi(J)$, except possibly for countably many ones, the intersections of both support lines through that point with $\pi(C)$ are points.

Replacing if necessary the point $(x_1^0, z^0)$ with an interior point of $\pi(J)$, without loss of generality we assume that
\beq\label{(del)}\hspace*{-3mm}
\text{the intersection of each line of support to }  \pi(C) \text{ through } (x_1^0, z^0) \text{ with } \pi(C)  \text{ is a point}
\eeq
and, additionally,
$$
u(x^0) - \ve < z^0 < u(x^0).
$$

Choose $\del > 0$ sufficiently small, so as the horizontal segment $I = [A, B]$ with the vertices
$$
A = (x^0, z^0) - (0, \del, 0) \qquad \text{and} \qquad B = (x^0, z^0) + (0, \del, 0)
$$
is contained in the interior of $C'$ and
$$
z^0 > u(x_1^0, x_2^0 + t) - \ve \qquad \text{for all} \quad |t| \le \del.
$$
Define
$$\tilde C = \text{{\conv}}(C \cup I)$$
and define the function $\widetilde{u}$ by the condition that $\tilde C$ is the epigraph of $\widetilde{u}$. We have $\widetilde{u} \le u$. Additionally, $\pl C \setminus \pl\tilde C \subset \pl C \setminus \pl C' \subset \text{graph}(u)\rfloor_\UUU$, and therefore,
\vspace{1mm}

(a) $\widetilde{u}\rfloor_{\Om\setminus\UUU} = u\rfloor_{\Om\setminus\UUU}$.
\vspace{1mm}

Each point of $\pl\tilde C$ is contained either in $\pl C$, or in a segment joining a point of $I$ with a point of $\pl C$. It follows that for any $x \in \Om$, either $\widetilde{u}(x) = u(x)$ or there exist $x^1 \in \Om$ and $x^2 = x^0 + (0,t),\, |t| \le \del$ such that the point $(x, \widetilde{u}(x))$ lies on the segment joining the points $(x^1, u(x^1))$ and $(x^2, z^0)$, that is, $x = \lam x^1 + (1-\lam) x^2$ and $\widetilde{u}(x) = \lam u(x^1) + (1-\lam) z^0$ for some $0 \le \lam < 1$.   
Taking into account that $u$ is convex and $z^0 > u(x^2) - \ve$, one obtains
$$
\widetilde{u}(x) > \lam u(x^1) + (1-\lam) (u(x^2) - \ve) \ge u(x) - (1-\lam) \ve.
$$
Thus, the following property is proved.
\vspace{1mm}

(b) $0 \le u - \widetilde{u} < \ve$.
\vspace{1mm}

There are two planes of support to $\pl\tilde C$ through each interior point of $I$, and this pair of planes does not depend on the choice of the point. Let them be designated as $\Pi_-$ and $\Pi_+$. The planes are of the form
\beq\label{Piplusminus}
\Pi_- :\, z - z^0 = \xi_- (x_1 - x_1^0) \quad \text{and} \quad \Pi_+ :\, z - z^0 = \xi_+ (x_1 - x_1^0),
\eeq
where $\xi_- < \xi_+$ are some real values (see Fig.~\ref{figPi}).
                          \begin{figure}[h]
\centering
\includegraphics[scale=0.3]{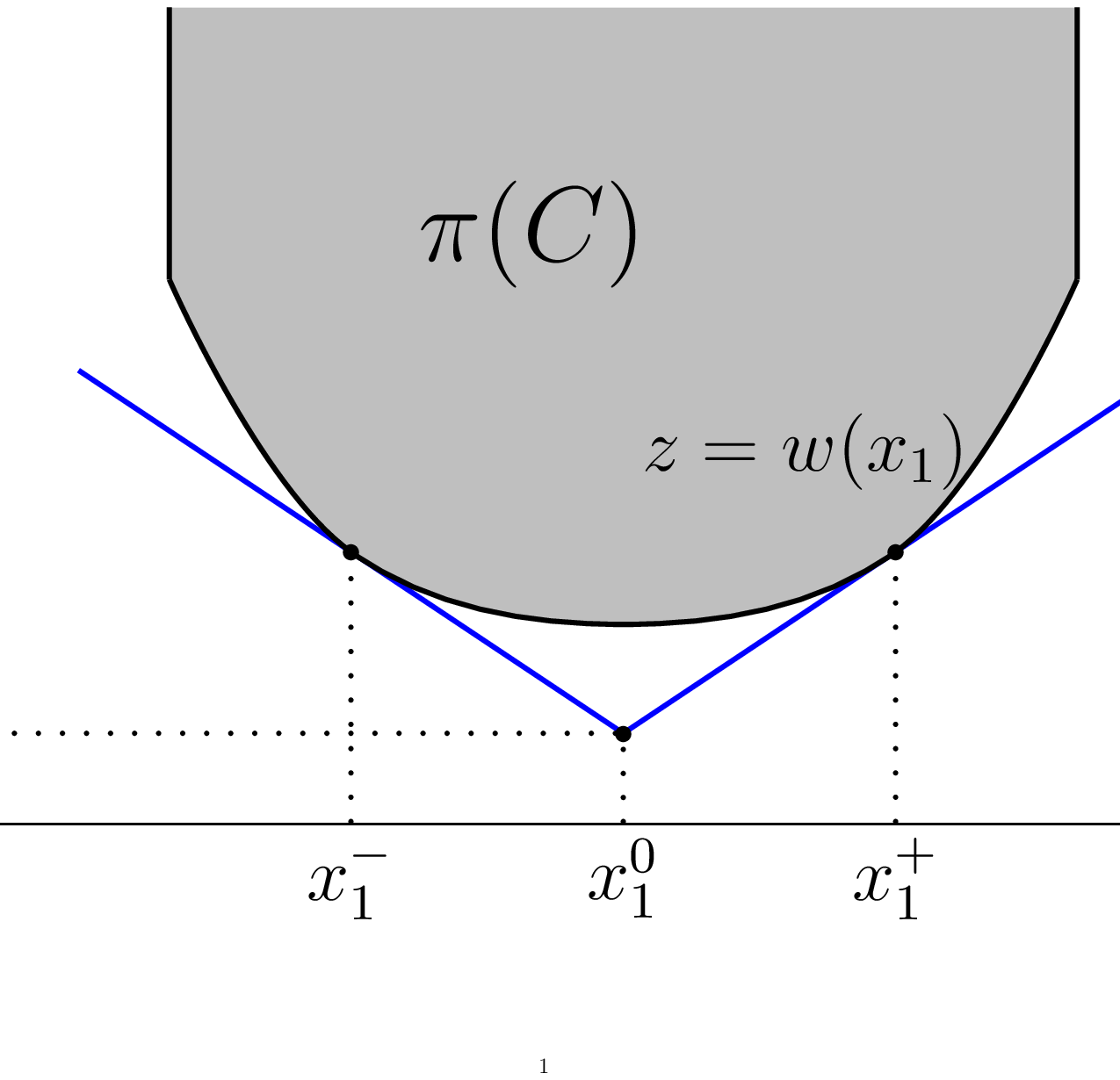}
\caption{The lines through the point $(x_1^0, z^0)$ are the orthogonal projections of $\Pi_-$ and $\Pi_+$ on the $(x_1,z)$-plane. The point $(x_1^0, z^0)$ is the projection of the segment $I$.}
\label{figPi}
\end{figure}
Condition \eqref{(del)} means that the intersection of each of these planes with $C$ is a line segment parallel to the $x_2$-axis (possibly degenerating to a point). Let these segments be denoted as
$$
I_\pm = \Pi_\pm \cap C = \big( x_1^\pm,\ x_2^\pm + [-a_\pm,\, a_\pm],\ z^0 + \xi_\pm (x_1^\pm - x_1^0) \big).
$$
Correspondingly, the intersection of each plane $\Pi_\pm$ with the graph of $\widetilde{u}$ is the graph of an affine function defined on a trapezoid. For the sake of the future use we denote these functions by $z_-(x)$ and $z_+(x)$ and provide their analytic description,
\beq\label{2affine}
z_-(x) - z^0 = \xi_- (x_1 - x_1^0), \ (x_1, x_2) \in T_- \ \ \text{and} \ \ z_+(x) - z^0 = \xi_+ (x_1 - x_1^0), \ (x_1, x_2) \in T_+,
\eeq
where $T_\pm$ is the trapezoid with the sides $S_0 = \big( x_1^0,\, x_2^0 + [-\del,\, \del] \big)$ and $S_\pm = \big( x_1^\pm,\, x_2^\pm + [-a_\pm,\, a_\pm])$, with $x_1^- < x_1^0 < x_1^+$ and $a_\pm \ge 0$; see Fig.~\ref{fig trapezoid}.

                    \begin{figure}[h]
\centering
\includegraphics[scale=0.2]{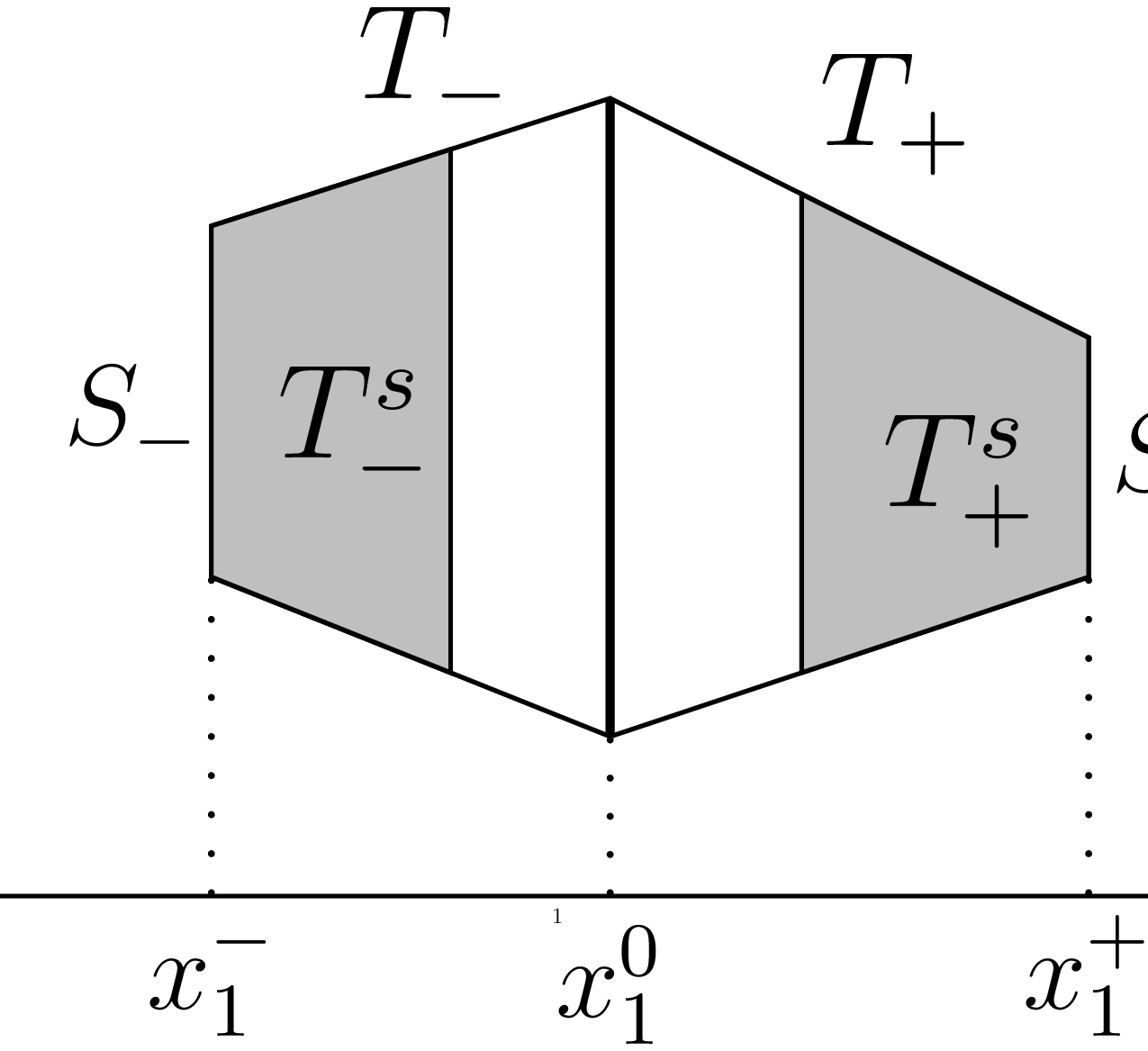}
\caption{The trapezoids $T_\pm$ and $T_\pm^s$ with $s = 0.6$.}
\label{fig trapezoid}
\end{figure}

All points of $S_\pm$ are regular and for any $x \in S_\pm$, $\nabla u(x) = (\xi_\pm, 0)$, and additionally,
$$
w'(x_1^\pm) = \xi_\pm \quad \text{and} \quad w(x_1^\pm) = z^0 + \xi_\pm (x_1^\pm - x_1^0),
$$
By condition \eqref{(beta)}, the restriction of $f$ on the line segment $[\xi_-,\, \xi_+] \times \{ 0 \}$ is strictly concave, that is,
\beq
f(\xi, 0), \ \xi \in [\xi_-,\, \xi_+] \quad \text{is a strictly concave function of one variable.}
\eeq

Using that for $x \in (x_1^-,\, x_1^+)$
$$
f(w'(x), 0)\ =\  f\Big( \frac{\xi_+ - w'(x)}{\xi_+ - \xi_-}\, \xi_- +  \frac{w'(x) - \xi_-}{\xi_+ - \xi_-}\, \xi_+, \ 0 \Big)
$$  $$
>\ \frac{\xi_+ - w'(x)}{\xi_+ - \xi_-}\, f(\xi_-, 0)\ +\ \frac{w'(x) - \xi_-}{\xi_+ - \xi_-}\, f(\xi_+, 0)
$$
and that
$$
\int_{x_1^-}^{x_1^+} w'(x)\, dx = w(x_1^+) - w(x_1^-) = \xi_+ (x_1^+ - x_1^0) - \xi_- (x_1^- - x_1^0),
$$
one obtains
\beq\label{bigineq}
\begin{split}
\int_{x_1^-}^{x_1^+} f(w'(x), 0)\, dx\ &>\ \frac{\xi_+(x_1^+ - x_1^-) - \int_{x_1^-}^{x_1^+} w'(x)\, dx}{\xi_+ - \xi_-}\ f(\xi_-, 0)\\
+\ \frac{\int_{x_1^-}^{x_1^+} w'(x)\, dx - \xi_-(x_1^+ - x_1^-)}{\xi_+ - \xi_-}\ f(\xi_+, 0) \ &=\ f(\xi_-, 0) (x_1^0 - x_1^-) + f(\xi_+, 0) (x_1^+ - x_1^0).
\end{split}
\eeq
We will need inequality \eqref{bigineq} in the end of subsection \ref{sub sg0}.
\vspace{2mm}

Define the family of convex sets $C_s,\, s \le 1$ by
$$
C_s =
\left\{
\begin{array}{ll}
(1-s)C + s\tilde C, & \text{if } 0 \le s \le 1;\\
C \cap \big[ (1-s)C + sA \big] \cap \big[ (1-s)C + sB \big], & \text{if } s < 0.
\end{array}\right.
$$
In particular, $C_0 = C$ and $C_1 = \tilde C$.

The following technical lemma will be proved in subsection \ref{sub lt2}.

\begin{lemma}\label{lt2}
$C_s$ is the epigraph of a convex function $u^{(s)}$ defined on $\Om$. There is a value $s_0 < 0$ such that for $s_0 \le s \le 1$,\, $u^{(s)}$ satisfies statements (a) and (b) of Theorem \ref{t3}.
\end{lemma}

   \subsubsection{The case $0 \le s \le 1$}\label{sub sg0}
Let us now study the values of the functional $F(u^{(s)})$ for $0 \le s \le 1$. We shall prove that $F(u^{(s)})$ is a polynomial of the $2^{\text{nd}}$ degree in $s$ and determine its coefficients.

Denote by $\Pi^n$,\, $\tilde\Pi^n$, and $\Pi^n_s$ the planes of support to $C$,\, $\tilde C$, and $C_s$, respectively, with the outward normal $n \in S^2$. In particular, $\Pi^n_0 = \Pi^n$ and $\Pi^n_1 = \tilde\Pi^n$. The following equations hold,
\beq\label{pies1}
   \Pi^n_s = (1-s) \Pi^n + s \tilde\Pi^n,   
   \eeq
   \beq\label{pies2}
\Pi^n_s \cap C_s = (1-s) (\Pi^n \cap C) + s (\tilde\Pi^n \cap \tilde C).
\eeq
We restrict ourselves to proving equation \eqref{pies2}, leaving equation \eqref{pies1} to the reader. We have
$$
\Pi^n \cap C = \{ r \in C :\, \langle r, n \rangle = \max_{\rho \in C} \langle \rho, n \rangle \}, \qquad \tilde\Pi^n \cap \tilde C = \{ r \in \tilde C :\, \langle r, n \rangle = \max_{\rho \in \tilde C} \langle \rho, n \rangle \},
$$
and since
$$
\max_{\rho \in C_s} \langle \rho, n \rangle = \max_{\rho_1 \in C,\, \rho_2 \in \tilde C} \langle (1-s)\rho_1 + s\rho_2,\, n \rangle = (1-s)\max_{\rho_1 \in C} \langle \rho_1, n \rangle + s\max_{\rho_2 \in \tilde C} \langle \rho_2, n \rangle,
$$
one comes to \eqref{pies2}:
$$
\Pi_s^n \cap C_s = \left\{ r  = (1-s) r_1 + s r_2 :\ r_1 \in C,\ r_2 \in \tilde C, \
\langle r_1, n \rangle = \max_{\rho_1 \in C} \langle \rho_1, n \rangle,\, \right. $$
$$ \left.
\langle r_2, n \rangle = \max_{\rho_2 \in \tilde C} \langle \rho_2, n \rangle \right\}
= (1-s) \big(\Pi^n \cap C\big) + s \big(\tilde\Pi^n \cap \tilde C\big).$$

For each $n$, the plane $\tilde\Pi^n$ may intersect or not intersect $\pl C$, and may intersect or not intersect $I$. If $\tilde\Pi^n$ intersects $I$, the following three cases are possible: $\tilde\Pi^n \cap I = A$,\, $\tilde\Pi^n \cap I = B$,\, $\tilde\Pi^n \cap I = I$. Additionally, $\tilde\Pi^n$ always intersects either $\pl C$ or $I$. Thus, there may be 7 cases:
\vspace{1mm}

 $\tilde\Pi^n \cap \pl C \ne \emptyset$ and $\tilde\Pi^n \cap I = \emptyset$;
\vspace{1mm}

 $\tilde\Pi^n \cap \pl C \ne \emptyset$ and $\tilde\Pi^n \cap I = A$; \qquad $\tilde\Pi^n \cap \pl C = \emptyset$ and $\tilde\Pi^n \cap I = A$;
\vspace{1mm}

 $\tilde\Pi^n \cap \pl C \ne \emptyset$ and $\tilde\Pi^n \cap I = B$; \qquad $\tilde\Pi^n \cap \pl C = \emptyset$ and $\tilde\Pi^n \cap I = B$;
\vspace{1mm}

$\tilde\Pi^n \cap \pl C \ne \emptyset$ and $\tilde\Pi^n \cap I = I$; \qquad $\tilde\Pi^n \cap \pl C = \emptyset$ and $\tilde\Pi^n \cap I = I$.
\vspace{1mm}

\hspace*{-6mm}Correspondingly, $S^2$ is the disjoint union of 7 sets,
$$
S^2 = \AAA_0 \sqcup \AAA_1^A \sqcup \AAA_1^B \sqcup \AAA_2^A \sqcup \AAA_2^B \sqcup \AAA_3 \sqcup \AAA_4,
$$
where $\AAA_0 = \{ n :\, \tilde\Pi^n \cap \pl C \ne \emptyset$ and $\tilde\Pi^n \cap I = \emptyset \}$;

$\AAA_1^A = \{ n :\, \tilde\Pi^n \cap \pl C = \emptyset$ and $\tilde\Pi^n \cap I = A \}$;

$\AAA_2^A = \{ n :\, \tilde\Pi^n \cap \pl C \ne \emptyset$ and $\tilde\Pi^n \cap I = A \}$;

$\AAA_1^B$ and $\AAA_2^B$ are defined in a similar way; \qquad
$\AAA_1 = \AAA_1^A \sqcup \AAA_1^B$; \, $\AAA_2 = \AAA_2^A \sqcup \AAA_2^B$;

$\AAA_3 = \{ n :\, \tilde\Pi^n \cap \pl C = \emptyset$ and $\tilde\Pi^n \cap I = I \}$;

$\AAA_4 = \{ n :\, \tilde\Pi^n \cap \pl C \ne \emptyset$ and $\tilde\Pi^n \cap I = I \}$.\\
The boundary of each of the bodies $C$, $\tilde C$, $C_s$ can be represented as the union,
$$
\pl C_s = \cup_{i=0}^4 \pl_i C_s \quad \text{with} \quad \pl_j C_s = \pl_j^A C_s \cup \pl_j^B C_s,  \ \ j = 1,\, 2,
$$
where
$$
\pl_i C_s = \Big( \cup_{n\in\AAA_i} \Pi^n_s \Big) \cap C_s, \ \, i = 0,\ldots,4, \quad
\pl_j^A C_s = \Big( \cup_{n\in\AAA_j^A} \Pi^n_s \Big) \cap C_s, \ \, j = 1,\, 2,
$$
and a similar representation holds for $\pl_j^B C_s$ (see Fig.~\ref{figPlC}).
                                 \begin{figure}[h]
\centering
\includegraphics[scale=0.19]{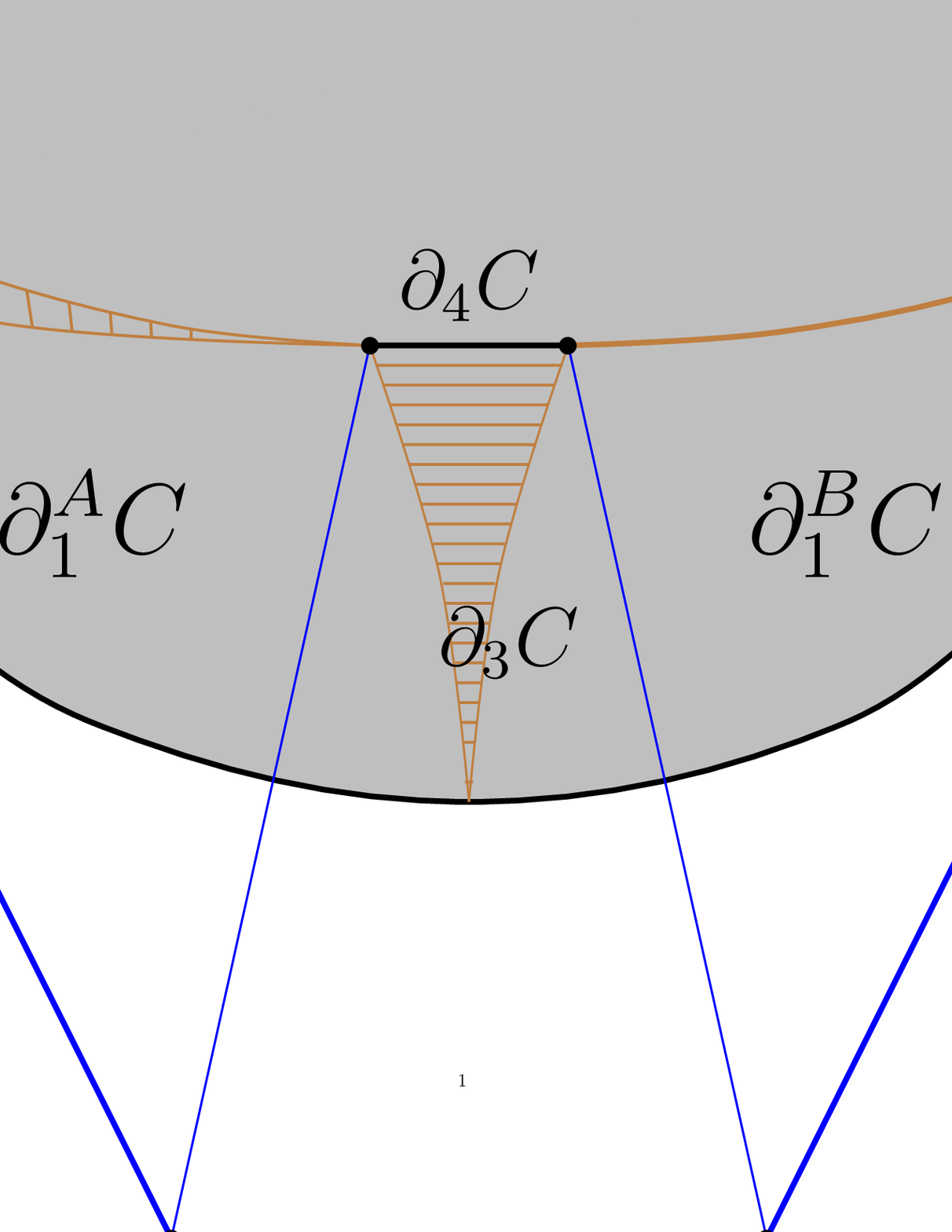}
\caption{In the figure, the sets $\pl_2^A C$ and $\pl_3 C$ are foliated by line segments, and $\pl_2^B C$ is a curve. The set $\pl_4 C$ is the union of segments $I_-$ and $I_+$; one of them is shown.}
\label{figPlC}
\end{figure}
We have $\pl_i C_0 = \pl_i C$ and $\pl_i C_1 = \pl_i \tilde C,\, i = 0,\ldots,4$. It follows from equation \eqref{pies2} that
\beq\label{pliCs}
\pl_i C_s = (1-s) \pl_i C + s \pl_i \tilde C, \ i = 0,\ldots,4, \quad \pl_j^A C_s = (1-s) \pl_j^A C + s \pl_j^A \tilde C, \ j = 1,\, 2,
\eeq
and the same is true for $\pl_j^B C_s$.

\begin{propo}\label{pro}
The sets $\pl_0 C_s,\, \pl_1^A C_s,\, \pl_1^B C_s,\, \pl_2^A C_s,\, \pl_2^B C_s,\, \pl_3 C_s,\, \pl_4 C_s$ for $0 \le s < 1$ are mutually disjoint.
\end{propo}

\begin{proof}
The sets $\pl_i C$, $i \ne 0$ belong to the graph of $u\rfloor_\UUU$, and therefore, do not contain singular points. Each point point $x \in \pl_i C_s$, $0 < s < 1$ admits a decomposition $x = (1-s) x_1 + s x_2$ with $x_1 \in \pl_i C$ and $x_2 \in \pl_i \tilde C$. If $x$ is a singular point of $C_s$, then $x_1$ is a singular point of $C$ and $x_2$ is a singular point of $\tilde C$. It follows that $x_1$ does not lie in $\pl_i C$ with $i \ne 0$, and so, the sets $\pl_i C_s$, $i \ne 0$ for $0 \le s < 1$ do not contain singular points.

Let us show that different sets from the list $\pl_i^* C_s$,\, $\pl_j^* C_s$, where each of the superscripts "$*$" should be removed or substituted with $A$ or $B$, are disjoint. Assume the contrary, that is, there exists a point $x \in \pl_i^* C_s \cap \pl_j^* C_s$; then there are planes of support to $C_s$ at $x$ with the outward normals from $\AAA_i^*$ and $\AAA_j^*$, hence $x$ is a singular point of $\pl C_s$. However, at least one of the subscripts $i,\, j$ is nonzero, and therefore, at least one of the sets $\pl_i^* C_s$,\, $\pl_j^* C_s$ does not contain singular points. The obtained contradiction proves the proposition.
\end{proof}

Taking $s = 1$, one sees that some of the sets $\pl_i \tilde C$,\, $i = 0,\ldots,4$ intersect. In particular, $\pl_4 \tilde C$ is the union of graphs of two affine functions $z_\pm(x)$ given by \eqref{2affine}, and $\pl_3 \tilde C = I$ is contained in $\pl_4 \tilde C$.

In Fig.~\ref{fig section} there is shown the section of $C$ by a plane through $AB$, which does not coincide with the planes $z - z^0 = \xi_\pm (x_1 - x_1^0)$. The plane is defined either by $z - z^0 = \xi (x_1 - x_1^0)$ with $\xi < \xi_-$ or $\xi > \xi_+$, or by $x_1 - x_1^0 = 0$.

                         \begin{figure}[h]
\centering
\includegraphics[scale=0.2]{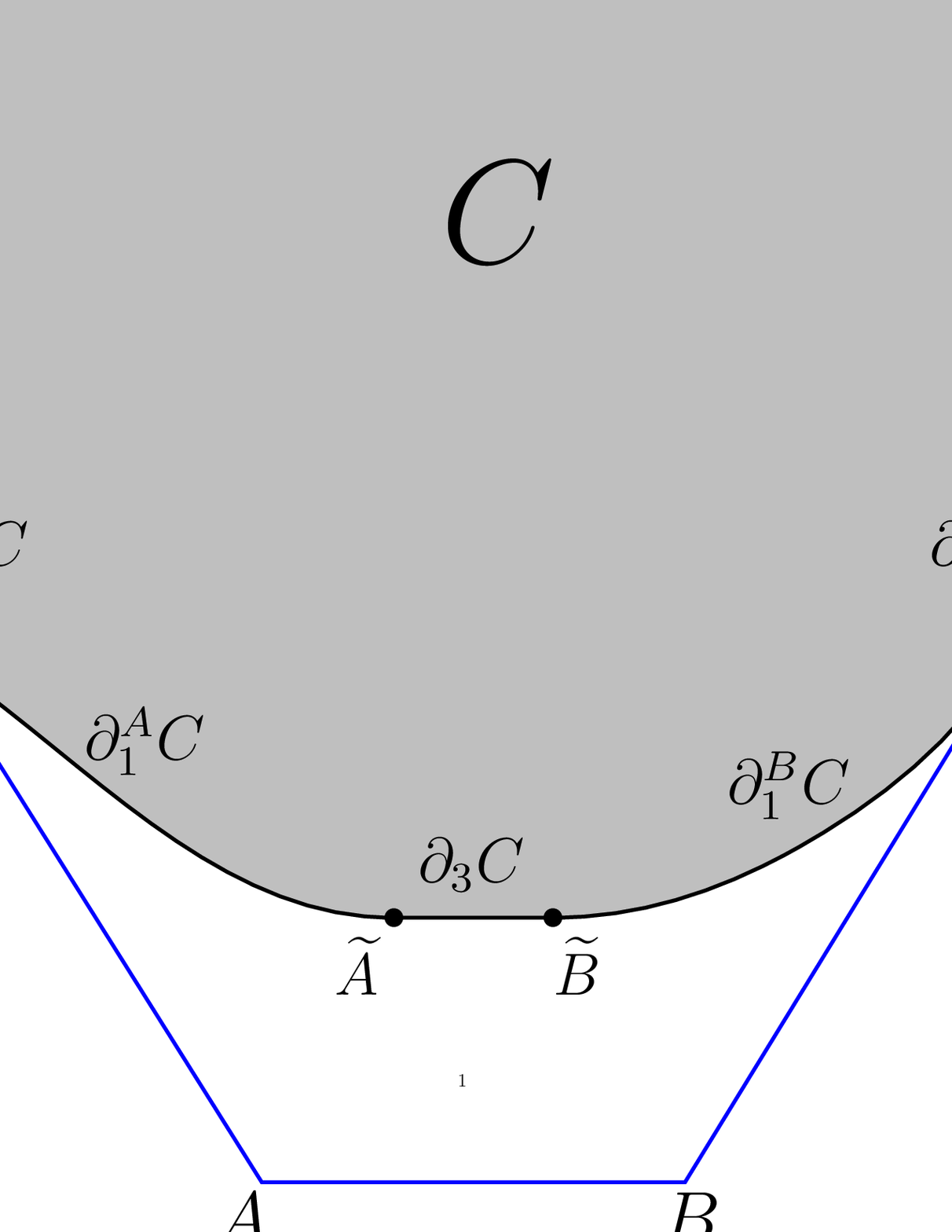}
\caption{The section of $C$ by a plane through $AB$ is shown. The section of the surface $\pl_0 C$ (not shown completely) is represented by the lines $A' \widehat{A}$ and $B'' \widehat{B}$. The sections of the surfaces $\pl_1^A C$ and $\pl_1^B C$ are the lines $A' \widetilde{A}$ and $B' \widetilde{B}$, respectively. The sections of $\pl_2^A C$ and $\pl_2^B C$ are, respectively, the point $A'$ and the closed line segment $B' B''$. The section of the surface $\pl_3 C$ is the closed line segment $\widetilde{A} \widetilde{B}$. The intersection of the plane with $\pl_4 C$ is empty.}
\label{fig section}
\end{figure}

In Fig.~\ref{fig section k<1}, the section of $C_s$ by the same plane is shown.

                        \begin{figure}[h]
\centering
\includegraphics[scale=0.2]{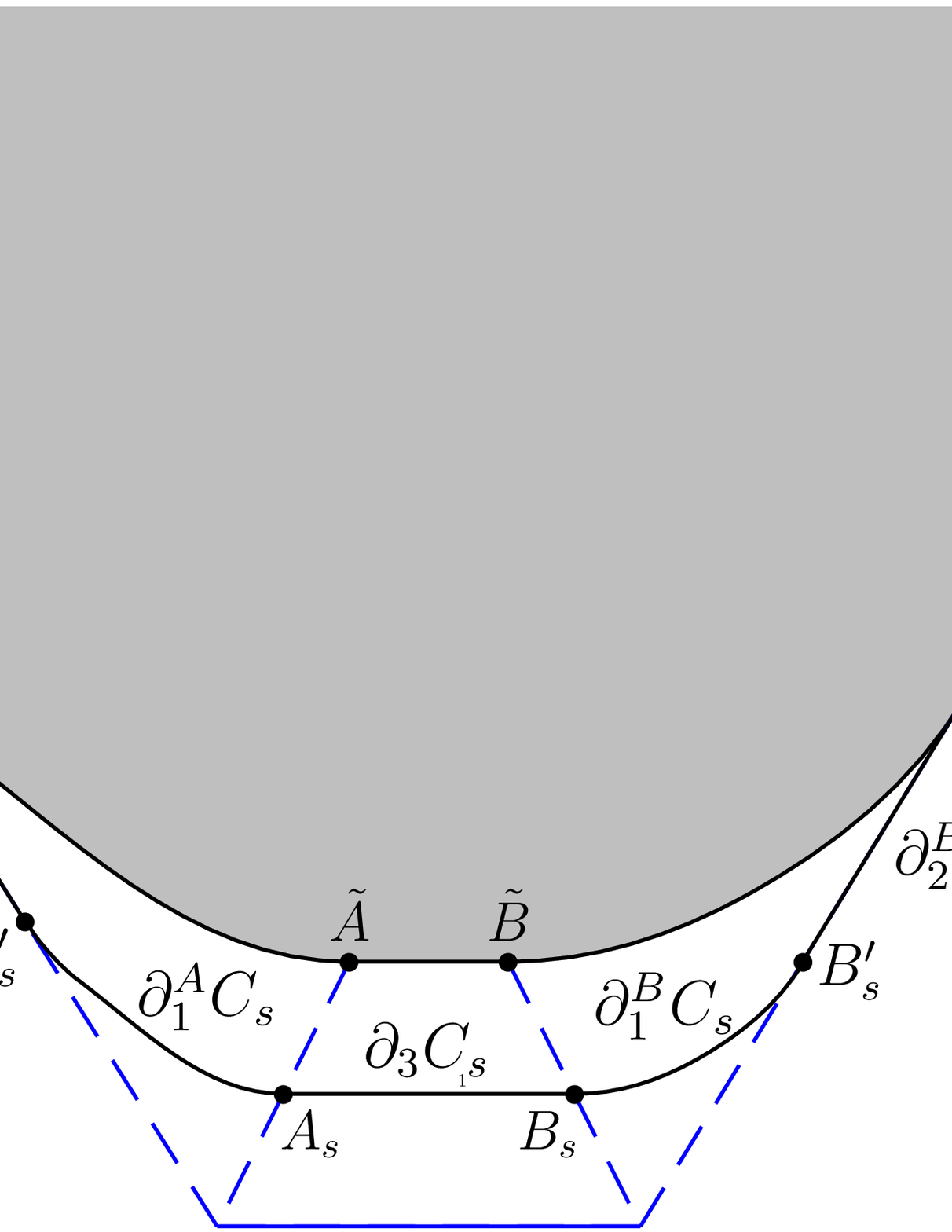}
\caption{The section of $C_s$ by a plane is shown, $s = 1/2$. The section of the surface $\pl_0 C_s$ coincides with the section of $\pl_0 C$. The closed line segments $A' {A}'_s$ and $B'' {B}'_s$ are, respectively, the sections of $\pl_2^A C_s$ and $\pl_2^B C_s$. The two curves and a closed line segment forming together the arc ${A}'_s {B}'_s$ represent the sections of $\pl_1^A C_s$,\, $\pl_1^B C_s$, and $\pl_3 C_s$.}
\label{fig section k<1}
\end{figure}

The section of $C_s$ by the plane $z - z^0 = \xi_+ (x_1 - x_1^0)$ is shown in Fig.~\ref{fig touch}. (The section by the plane $z - z^0 = \xi_- (x_1 - x_1^0)$ looks similar to it.) The plane is tangent to both sets $C$ and $C_s$. The intersection of the plane with $C$ is a line segment, and the intersection with $C_s$ is a trapezoid (they may degenerate to a point and a triangle, respectively).
                             \begin{figure}[h]
\centering
\includegraphics[scale=0.11]{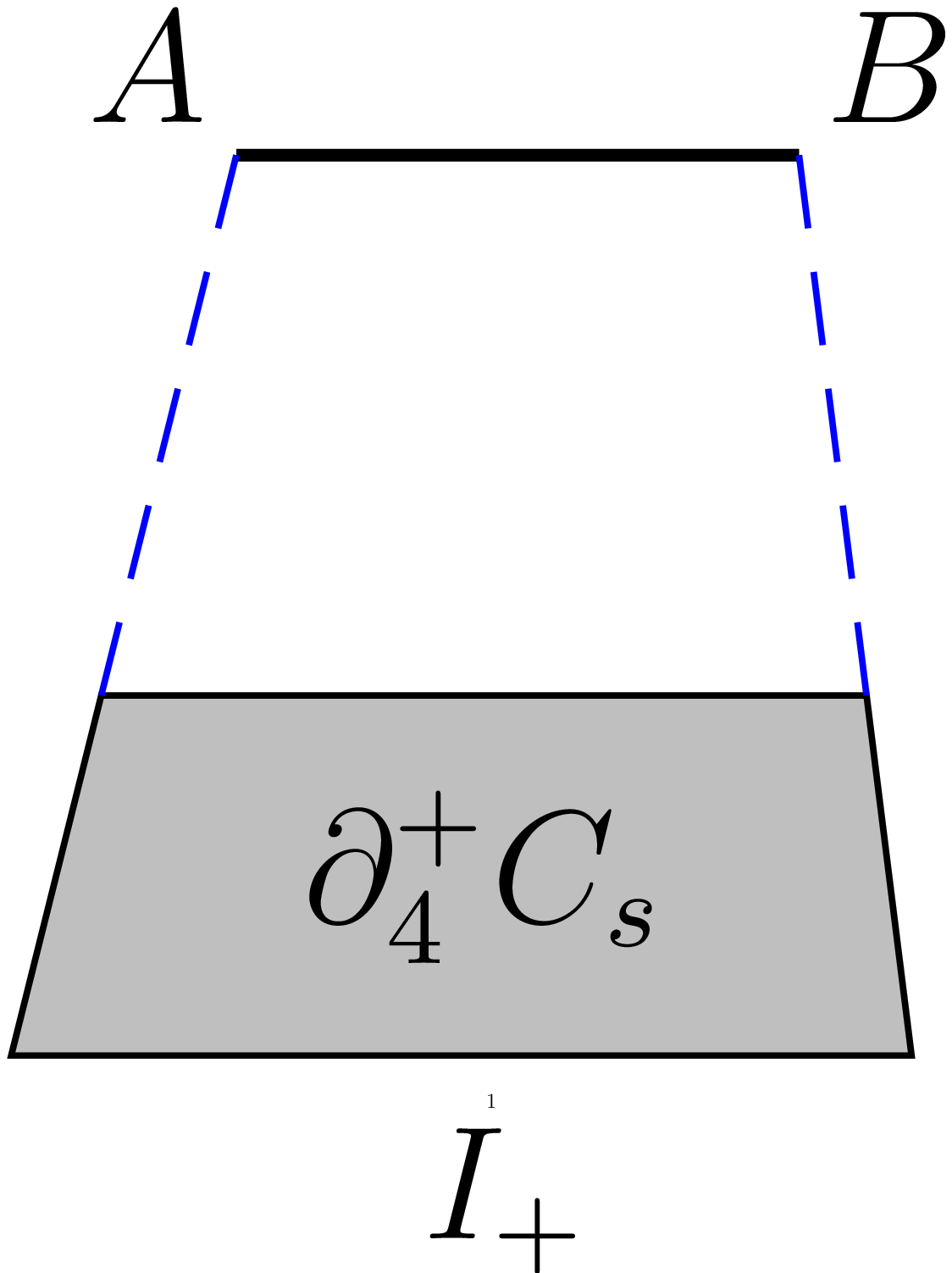}
\caption{The section of $C_s$ with $0 < s < 1$ by the plane $z - z^0 = \xi_+ (x_1 - x_1^0)$ is shown. The intersection of the plane with $C$ is the segment $I_+ = \pl_4^+ C$, and the intersection with $C_s$ is the trapezoid $\pl_4^+ C_s$. }
\label{fig touch}
\end{figure}

Later on we will consider intersections of the sets $C_s$ with the planes through $AB$. Having this in mind, denote by $\Pi[x_1]$ the half-plane bounded by line $AB$ and containing the line $(x_1, \RRR, w(x_1))$,\, $x_1 \in [x_1^-,\, x_1^+]$. In particular, the half-planes $\Pi[x_1^\pm]$ are contained, correspondingly, in the planes $\Pi_\pm$ defined by \eqref{Piplusminus}.

We are going to determine the resistance (the value of the functional $F$) produced by each of the sets $\pl_i C_s$,\, $0 \le s < 1$,\, $i = 0,\ldots,4$. Let us consider them separately.
\vspace{2mm}

{\bf [0]} If $n \in \AAA_0$ then $\tilde\Pi^n = \Pi^n = \Pi^n_s$ for all $0 \le s \le 1$.
 The sets $\pl_0 C_s$ for all $0 \le s \le 1$ coincide with $\pl_0 C$. The corresponding resistance does not depend on $s$,
\beq\label{pl_0}
F(\pl_0 C_s) = F(\pl_0 C) = a_0.
\eeq
\vspace{1mm}

{\bf [1]} If $n \in \AAA_1^A$ then $\tilde\Pi^n \cap \tilde C = A$, and if $n \in \AAA_1^B$ then $\tilde\Pi^n \cap \tilde C = B$, hence $\pl_1^A \tilde C = A$ and $\pl_1^B \tilde C = B$. By Proposition \ref{pro}, the corresponding sets $\pl_1^A C_s$ and $\pl_1^B C_s$ are disjoint, and by \eqref{pliCs},
$$
\pl_1^A C_s = (1-s) \pl_1^A C + s A, \qquad \pl_1^B C_s = (1-s) \pl_1^B C + s B.
$$
Thus, $\pl_1^A C_s$ and $\pl_1^B C_s$ are homothetic, respectively, to $\pl_1^A C$ and $\pl_1^B C$ with the ratio $1-s$ and with the centers at $A$ and $B$, and therefore, $F(\pl_1^A C_s) = (1-s)^2 F(\pl_1^A C)$ and $F(\pl_1^B C_s) = (1-s)^2 F(\pl_1^B C)$. Denoting $a_1 = 2F(\pl_1 C) = 2(F(\pl_1^A C) + F(\pl_1^B C))$, we obtain
\beq\label{pl_1}
F(\pl_1 C_s) = F(\pl_1^A C_s) + F(\pl_1^B C_s) = \frac{a_1}{2} (1-s)^2.
\eeq
\vspace{1mm}

{\bf [2]} Now consider $\AAA_2 = \AAA_2^{A} \cup \AAA_2^{B}$.
       %
The set $\pl_2^A \tilde C$ is the union of line segments with one endpoint at $A$. Each crossing half-plane $\Pi[x_1]$,\, $x_1 \in (x_1^-,\, x_1^+)$ contains such a segment; let it be designated as $[A,\, A''(x_1)]$. It is the union $[A,\, A''(x_1)] = [A,\, A'(x_1)) \cup [A'(x_1),\, A''(x_1)]$, where $[A,\, A'(x_1)) \cap C = \emptyset$ and $[A'(x_1),\, A''(x_1)] \subset C$. The set $\pl_2^A C$ is the union of the corresponding segments $[A''(x_1),\, A'(x_1)]$,\, $\pl_2^A \tilde C$ is the union of segments $[A''(x_1),\, A]$, and $\pl_2^A C_s$ is the union of segments $[A''(x_1),\, A_s(x_1)]$, where $A_s(x_1) = (1-s) A'(x_1) + s A$.

Note that $(A_s(x_1),\, A] = (1-s) (A'(x_1),\, A] + s A$, hence the segment $[A''(x_1),\, A_s(x_1)]$ is the set-theoretic difference
$$
[A''(x_1),\, A_s(x_1)] = [A''(x_1),\, A] \setminus \big( (1-s) (A'(x_1),\, A] + s A \big) \big].
$$


Thus, $\pl_2^A C_s$ is the set-theoretic difference of of $\pl_2^A \tilde C$ and the set homothetic to $\pl_2^A \tilde C \setminus \pl_2^A C$ with ratio $1-s$ and the center $A$,
$$
\pl_2^A C_s = \pl_2^A \tilde C \setminus \big( (1-s) (\pl_2^A \tilde C \setminus \pl_2^A C) + s A \big).
$$
Hence
$$
F(\pl_2^A C_s) = F(\pl_2^A \tilde C) - (1-s)^2 (F(\pl_2^A \tilde C) - F(\pl_2^A C)),
$$
and a similar relation holds for $F(\pl_2^B C_s)$.
By Proposition \ref{pro}, the sets $\pl_2^A C_s$ and $\pl_2^B C_s$ are disjoint.
Thus, denoting
$${a_2} = 2( F(\pl_2^A \tilde C) - F(\pl_2^A C) ) + 2( F(\pl_2^B \tilde C) - F(\pl_2^B C) ) = 2( F(\pl_2 \tilde C) - F(\pl_2 C) )$$
 and
 $$b_2 = F(\pl_2^A \tilde C) + F(\pl_2^B \tilde C) = F(\pl_2 \tilde C),$$
  we obtain
\beq\label{pl_2}
F(\pl_2 C_s) = b_2 - \frac{a_2}{2} (1-s)^2.
\eeq
\vspace{1mm}

{\bf [3]} We have $\pl_3 \tilde C = I$ and $\AAA_3 = \{ n : \tilde\Pi^n \cap \tilde C = I \}$. The set $\AAA_3$ is the smaller arc of the great circle in the plane $n_2 = 0$ with the endpoints $n_\pm = (\xi_\pm, 0, -1)/\sqrt{1 + \xi_\pm^2}$.
The image of the set $\pl_3 C = \big( \cup_{n \in \AAA_3} \Pi^n \big) \cap C$ under the map $\pi : (x_1, x_2, z) \mapsto (x_1, z)$ is the graph of  $w\rfloor_{(x_1^-,\, x_1^+)}$; recall that $w(x_1) = \inf_{x_2} u(x_1,x_2)$ and $w'(x_1^\pm) = \xi_\pm$.

The intersection of the set $\pl_3 C$ with each half-plane $\Pi[x_1]$ is a line segment parallel to $I$ (maybe degenerating to a point); let it be the segment $[A(x_1),\, B(x_1)]$ co-directional with $[A,\, B]$. (In Fig.~\ref{fig section} this segment is designated as $[\widetilde{A},\, \widetilde{B}]$.) Correspondingly, the set $\pl_3 C$ is the union of these segments,
$$
\pl_3 C = \cup_{x_1 \in (x_1^-,\, x_1^+)} [A(x_1),\, B(x_1)],
$$
and $\pl_s C_s$ is
$$
\pl_3 C_s = (1-s) \pl_3 C + s I = \cup_{x_1 \in (x_1^-,\, x_1^+)} [ A_s(x_1),\, B_s(x_1) ],
$$
where $A_s(x_1) = (1-s) A(x_1) + s A$ and $B_s(x_1) = (1-s) B(x_1) + s B$. (In particular, $A_0(x_1) = A(x_1)$ and $B_0(x_1) = B(x_1)$.)

Denote by $L(x_1)$ the length of $ [A(x_1),\, B(x_1)]$; then the length of $[ A_s(x_1),\, B_s(x_1) ]$ is $(1-s) L(x_1) + s |I|$. Recall that $|I| = 2\del$.

The image of $\pl_3 C_s$ under the map $\pi : (x_1, x_2, z) \mapsto (x_1, z)$ is $(1-s)$\,graph$\big( w\rfloor_{(x_1^-,\, x_1^+)} \big) + s (x_1^0, z^0)$; that is, it is the homothety of $\text{graph}\big( w\rfloor_{(x_1^-,\, x_1^+)} \big)$ with the center $(x_1^0, z^0)$ and ratio $1-s$, and is composed of points $\big( (1-s)x_1 + sx_1^0,\ (1-s)w(x_1) + sz^0 \big)$,\, $x_1 \in (x_1^-,\, x_1^+)$.
The gradient of the function $u^{(s)}$ at a preimage of such a point is $(w'(x_1), 0)$, hence
$$
F(\pl_3 C_s) = \int_{(1-s)x_1^- +sx_1^0}^{(1-s)x_1^+ +sx_1^0} f\Big(w'\Big(\frac{t-sx_1^0}{1-s}\Big), 0\Big)\, \Big[(1-s) L\Big( \frac{t-sx_1^0}{1-s} \Big) + s |I| \Big]\, dt $$
\beq\label{pl_3}
= (1-s) \int_{x_1^-}^{x_1^+} f(w'(x), 0)\, \big[ (1-s) L(x) + 2s\del \big]\, dx = a_3 s(1-s) + \frac{b_3}{2}\, (1-s)^2,
\eeq
\beq\label{a3}
\text{where} \ \, b_3 = 2\int_{x_1^-}^{x_1^+} f(w'(x), 0) L(x)\, dx \ \, \text{and} \ \, a_3 = 2\del \int_{x_1^-}^{x_1^+} f(w'(x), 0)\, dx.
\eeq
\vspace{2mm}

{\bf [4]} $\AAA_4$ is the set of two vectors, $\AAA_4 = \{ n_-, n_+ \}$; recall that $n_\pm = (\xi_\pm, 0, -1)/\sqrt{1 + \xi_\pm^2}$.

The sets $\tilde\Pi^{n_-} \cap \tilde C = \pl_4^- \tilde C$ and $\tilde\Pi^{n_+} \cap \tilde C = \pl_4^+ \tilde C$ are graphs of the affine functions $z_\pm(x)$ defined by \eqref{2affine}, and $\pl_4^- \tilde C \cup \pl_4^+ \tilde C = \pl_4 \tilde C$. The planes $\Pi^{n_\pm} = \tilde\Pi^{n_\pm}$ coincide with the panes $\Pi_\pm$ defined by \eqref{Piplusminus}, and $\pl_4 C = I_- \cup I_+$.

Further, $\pl_4 C_s = \pl_4^- C_{s} \cup \pl_4^+ C_{s}$, where each set $\pl_4^\pm C_{s}$ is the graph of the restriction of the function $z = z_\pm(x)$ on the trapezoid $T_\pm^s$ with a base $(x_1^\pm,\ x_2^\pm + [-a_\pm, a_\pm])$ and the lateral sides contained in the lateral sides of $T_\pm$, and with the height $s |x_1^0 - x_1^\pm|$; see Fig.~\ref{fig trapezoid}. The length of the other base is $2(1-s)a_\pm + 2s\del$. The slope of the planar set $\pl_4^\pm C_{s}$ is $\xi_\pm$, and the area of its projection on the $x$-plane equals
 $$
 \Del_4^\pm = s |x_1^0 - x_1^\pm| \big[ (2-s) a_\pm + s \del \big] = (1 - (1-s)^2) a_\pm |x_1^0 - x_1^\pm| + s^2 \del |x_1^0 - x_1^\pm|.
 $$
It follows that
\beq\label{pl_4}
F(\pl_4 C_s) = F(\pl_4^- C_{s}) +F(\pl_4^+ C_{s}) = \frac{a_4}{2} s^2 + b_4(1 - (1-s)^2),
\eeq
 where
\beq\label{a4}
a_4 = 2\del \big( f(\xi_-, 0) (x_1^0 - x_1^-) + f(\xi_+, 0) (x_1^+ - x_1^0) \big), \ \ b_4 = f(\xi_-, 0) (x_1^0 - x_1^-) a_- + f(\xi_+, 0) (x_1^+ - x_1^0) a_+.
\eeq
\vspace{2mm}

As a result of calculation in {\bf [0]}, {\bf [1]}, {\bf [2]}, {\bf [3]}, {\bf [4]} we conclude that the resistance of $u^{(s)}$ for $0 \le s < 1$ is a quadratic function in $s$,
$$
F(u^{(s)}) = F(\pl_0 C_s) + F(\pl_1 C_s) + F(\pl_2 C) + F(\pl_3 C_s) + F(\pl_4 C_s)
$$
 \beq\label{F kless1}
 = c_0 + \frac{c_1}{2} (1-s)^2 + a_3 s(1-s) + \frac{a_4}{2} s^2,
\eeq
where $c_0 = a_0 + b_2 + b_4$ and $c_1 = a_1 - a_2 + b_3 - 2b_4$. Hence the first and the second derivatives of $F(u^{(s)})$ for $0 < s < 1$ are equal to
$$
\frac{d}{ds} F(u^{(s)}) = c_1 (s - 1) + a_3(1 - 2s) + a_4 s, \qquad \frac{d^2}{ds^2} F(u^{(s)}) = c_1 - 2a_3 + a_4,
$$
and the right derivative at $s=0$ is
$$
\frac{d}{ds}\bigg\rfloor_{s=0^+} F(u^{(s)}) = a_3 - c_1.
$$

The following lemma serves to determine the left derivative of $F(u^{(s)})$ at $s=0$.

\begin{lemma}\label{ls<0}
$$
F(u^{(s)}) =  c_0 + \frac{c_1}{2} + a_3 s - c_1 s + o(s) \quad \text{\rm as} \ s \to 0^-.
$$
\end{lemma}

The crucial property used in the proof of this lemma is that all planes of support to $C$ through $A$ and $B$ are tangent to $C$, that is, all points in the intersection of $C$ with these planes are regular.

Before proving this lemma, let us finish the proof of Theorem \ref{t3}.

It follow from Lemma \ref{ls<0} that the derivative of $F(u^{(s)})$ at $s = 0$ exists, and
\beq\label{derivative k=1}
 \frac{d}{ds} F(u^{(s)})\Big\rfloor_{s=0} = a_3 - c_1.
\eeq

Now consider two cases. If $\frac{d}{ds} F(u^{(s)})\rfloor_{s=0} \ne 0$ then for $s$ sufficiently small (either positive or negative), the function $u_1 = u^{(s)}$ satisfies statement (c) of Theorem \ref{t3}. Statements (a) and (b) are guaranteed by Lemma \ref{lt2}, and so, Theorem \ref{t3} is proved.

If $\frac{d}{ds} F(u^{(s)})\rfloor_{s=0} = 0$ then by \eqref{derivative k=1} we have $a_3 = c_1$. From inequality \eqref{bigineq} and the definition of $a_3$ and $a_4$ in \eqref{a3} and \eqref{a4} one obtains $a_3 > a_4$, and therefore,
$$
\frac{d^2}{ds^2}\bigg\rfloor_{s=0^+} F(u^{(s)}) = - a_3 + a_4 < 0.
$$
It follows that for $s>0$ sufficiently small, $u_1 = u^{(s)}$ satisfies statement (c) of Theorem \ref{t3}.

Theorem \ref{t3} is completely proved.

    \subsubsection{The case $s < 1$}

It remains to prove Lemma \ref{ls<0}. The proof is quite difficult and involves the study of properties of the surface $C_s$ for $s < 0$.

The intersections of a plane through $AB$ with the sets $\pl_2^A C$, $\pl_2^B C$, $\pl_3 C$ are line segments, which may degenerate to points. The case when all these segments are points ($A'$, $B'$, and $A_0$, respectively) is represented in Fig.~\ref{fig section AB}, and the case when all three segments are non-degenerate segments ($[A',\, A'']$, $[B',\, B'']$, and $[A_0,\, B_0]$) is shown in Fig.~\ref{fig seckgr1}.

Denote
\beq\label{del}
\del_0 C = \pl_0 C \cup \pl_2 C \cup \pl_4 C, \quad \del^A C = \pl_1^A C \cup \pl_3 C, \ \ \text{and} \ \ \del^B C = \pl_1^B C.
\eeq
These sets can be characterized as follows. $\del_0 C$ is the set of points $r \in \pl C$ such that the closed half-space bounded by a plane of support to $C$ at $r$ containing $C$ also contains $I$. $\del^A C$ is the set of points $r \in \pl C$ such that the corresponding closed half-space does not contain $A$ and, additionally, the vector $\overrightarrow{AB}$ points inside this subspace or is parallel to it, that is, $\langle \overrightarrow{AB},\, n_r \rangle \le 0$. $\del^B C$ is the set of points $r \in \pl C$ such that the corresponding closed half-space does not contain $B$ and the vector $\overrightarrow{AB}$ points outside this subspace, that is, $\langle \overrightarrow{AB},\, n_r \rangle > 0$. Observe that there is a certain asymmetry between the sets $\del^A C$ and $\del^B C$.

The proof of Lemma \ref{ls<0} is based on the following Propositions \ref{propo3}, \ref{propo5.0}, and \ref{propo5}. The proofs of Propositions \ref{propo3} and \ref{propo5} are given in Section \ref{sec technical}.

\begin{propo}\label{propo3}
For $s < 0$,\, $\pl C_s$ is the union of three sets,
$$
G_0^s = \del_0 C \cap \big[ (1-s) {C} + sA \big] \cap \big[ (1-s) {C} + sB \big],
$$
$$
G_1^s = C \cap \big[ (1-s) \del^A C + sA \big] \cap \big[ (1-s)C + sB \big],
$$
$$
G_2^s = C \cap \big[ (1-s) {C} + sA \big] \cap \big[ (1-s) \del^B C + sB \big],
$$
and $F(u^{(s)})$ is the sum of resistances of these sets,
$$ F(u^{(s)}) = F(G_0^s) + F(G_1^s) + F(G_2^s). $$
\end{propo}

                                  \begin{figure}[h]
\centering
\includegraphics[scale=0.25]{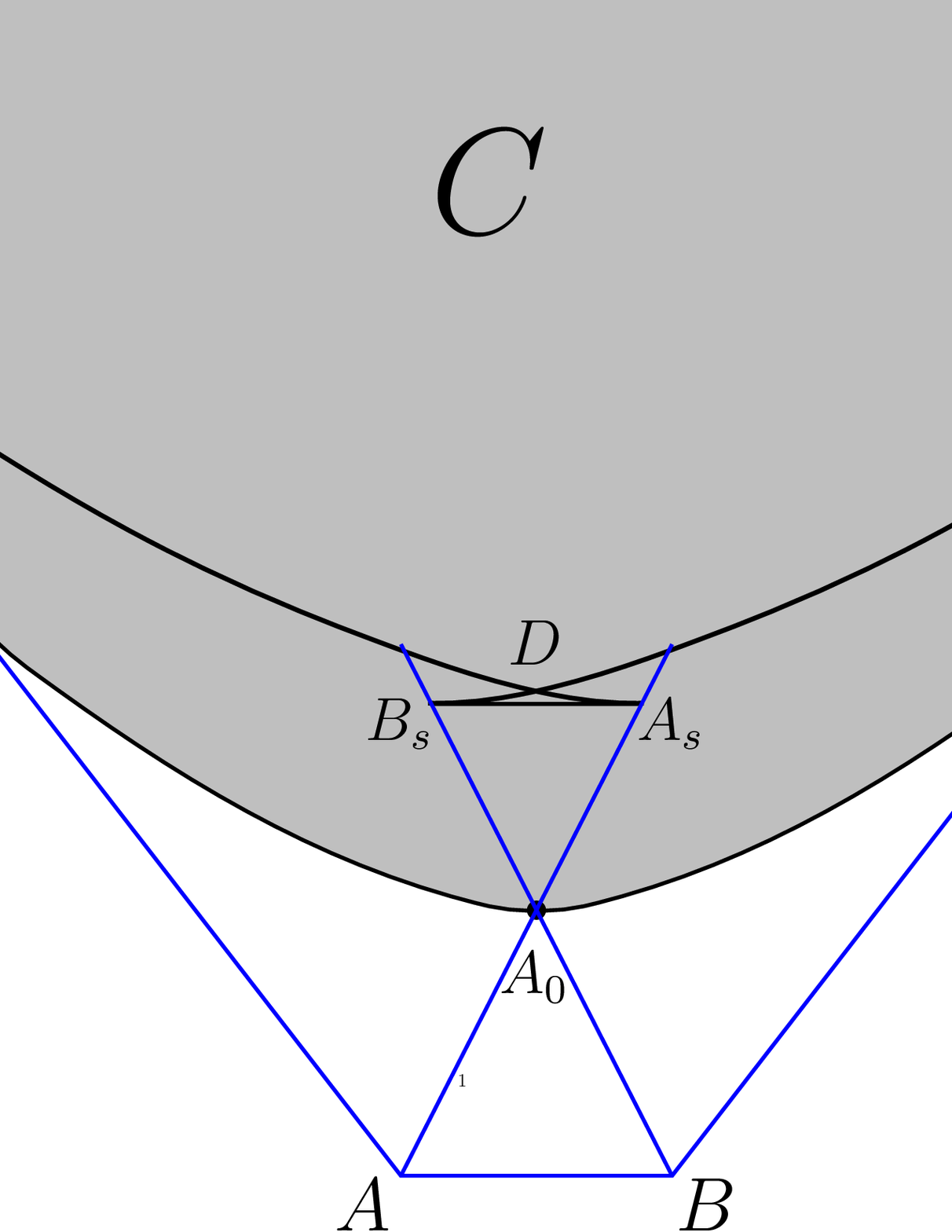}
\caption{The sections of $C$ and $C_s$ with $s < 0$ by a plane through $AB$ are shown. The (lower part of the) section of $C$ is shown lightgray. The section of $C_s$ is bounded below by the union of arcs $A'_s{A}_s$ and $B'_s{B}_s$. The arc $A'_s{A}_s$ is a part of the section of the surface $(1-s) \pl C + s A$, and the arc $B'_s{B}_s$ is a part of the section of the surface $(1-s) \pl C + s B$.}
\label{fig section AB}
\end{figure}

                                       \begin{figure}[h]
\centering
\includegraphics[scale=0.25]{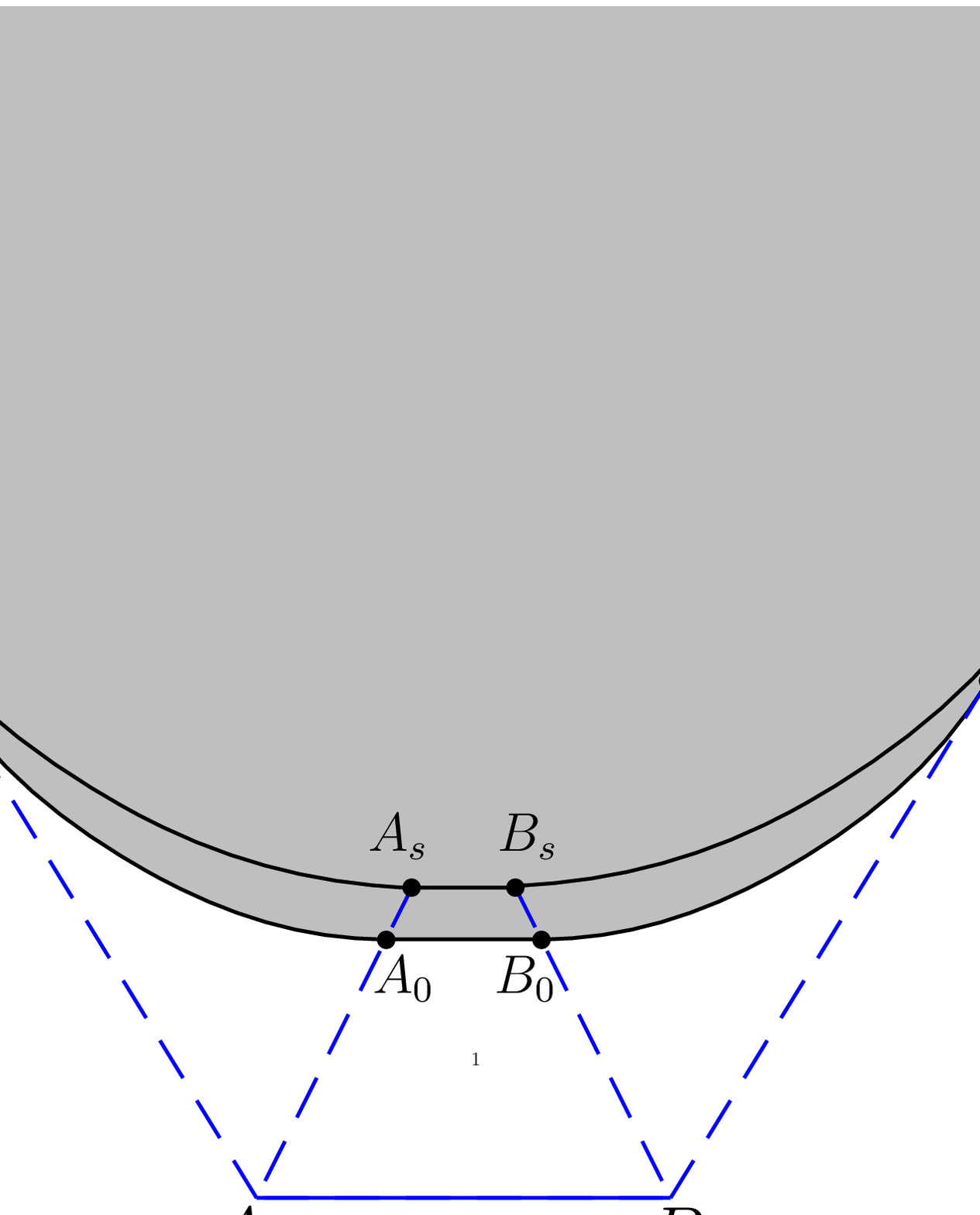}
\caption{The sections of $C$ and $C_s$ with $s < 0$ by a plane through $AB$ are shown. In this case the sections of $\pl_2^A C$, $\pl_2^B C$, $\pl_3 C$ are non-degenerate segments $[A',\, A'']$, $[B',\, B'']$, and $[A0,\, B_0]$, respectively. The section of $C$ is shown lightgray. The section of $C_s$ is bounded below by the curve $A'_s A_s B_s B'_s$.}
\label{fig seckgr1}
\end{figure}

The sections of the sets $G_0^s$,\, $G_1^s$, and $G_2^s$ by a plane through $AB$ are shown in Figures~\ref{fig section AB} and \ref{fig seckgr1}. In the former figure these sections are, respectively, the union of curves $A''' \widehat A$ and $B''' \widehat B$, the curve $A''' D$, and the curve $D B'''$. In the latter figure they are, respectively, the union of curves $A'_s \widehat A$ and $B'_s \widehat B$, $A'_s B_s$, and $B_s B_s'$.

Using the identity for three arbitrary sets $\mathcal{A},\, \mathcal{B}$, and $\mathcal{C}$,
$$
F(\mathcal{A} \cap \mathcal{B} \cap \mathcal{C}) =
F(\mathcal{A}) - F(\mathcal{A} \setminus \mathcal{B}) - F(\mathcal{A} \setminus \mathcal{C})
+ F\big( \mathcal{A} \setminus (\mathcal{B} \cup \mathcal{C}) \big),
$$
one obtains
$$
F(G_0^s) = F(\del_0 C) - F\big(\del_0 C \setminus \big[ (1-s)C + sA \big] \big) - F\big( \del_0 C \setminus \big[ (1-s)C + sB \big] \big)
$$
\beq\label{FG1s}
+  F\big( \del_0 C \setminus \big[ (1-s)C + s(A \cup B) \big] \big);
\eeq
$$
F(G_1^s) = F\big( (1-s) \del^A C + sA \big) -  F\big( \big[ (1-s) \del^A C + sA \big] \setminus C \big)
- F\big( \big[ (1-s) \del^A C + sA \big] \setminus \big[ (1-s)C + sB \big] \big) $$
\beq\label{FG2s}
+ F\big( \big[ (1-s) \del^A C + sA \big] \setminus \big( C \cup \big[ (1-s)C + sB \big] \big) \big);
\eeq
$$
F(G_2^s) = F\big( (1-s) \del^B C + sB \big) -  F\big( \big[ (1-s) \del^B C + sB \big] \setminus C \big)
- F\big( \big[ (1-s) \del^B C + sB \big] \setminus \big[ (1-s)C + sA \big] \big) $$
\beq\label{FG3s}
+ F\big( \big[ (1-s) \del^B C + sB \big] \setminus \big( C \cup \big[ (1-s)C + sA \big] \big) \big).
\eeq
Observe that the sections of the sets $\del_0 C \setminus \big[ (1-s)C + sA \big]$ and $\del_0 C \setminus \big[ (1-s)C + sB \big]$ in Fig.~\ref{fig section AB} are, respectively, the curves $A'A'''$ and $B'B'''$;
the sections of $\big[ (1-s) \del^AC + sA \big] \setminus C$ and $\big[ (1-s) \del^AC + sA \big] \setminus \big[ (1-s)C + sB \big]$ are, respectively, the curves $A_s'A'''$ and $A_s D$;
the sections of $\big[ (1-s) \del^B C + sB \big] \setminus C$ and $\big[ (1-s) \del^B C + sB \big] \setminus \big[ (1-s)C + sA \big]$ are, respectively, the curves $B_s'B'''$ and $B_s D$.

We are going to estimate the terms in the right hand sides of \eqref{FG1s}, \eqref{FG2s}, and \eqref{FG3s}. Since $(1-s) \del^A C + sA$ is the homothety of $\del^A C$ with the center at $A$ and ratio $1-s$, one has $F\big( (1-s) \del^A C + sA \big) = (1-s)^2 F(\del^A C)$, and similarly, $F\big( (1-s) \del^B C + sB \big) = (1-s)^2 F(\del^B C)$. Now, using formula \eqref{del} and setting $s=0$ in formulas \eqref{pl_0}, \eqref{pl_1}, \eqref{pl_2}, \eqref{pl_3}, \eqref{pl_4}, one comes to the following proposition.

\begin{propo}\label{propo5.0}
$$
F(\del_0 C) = F(\pl_0 C) + F(\pl_2 C) + F(\pl_4 C) = a_0 + b_2 - \frac{a_2}{2};
$$
$$
F\big( (1-s) \del^A C + sA \big) + F\big( (1-s) \del^B C + sB \big)
$$ $$
= (1-s)^2 \big( F(\pl_1^A C) + F(\pl_1^B C) + F(\pl_3 C) \big) = (1-s)^2\, \frac{a_1 +b_3}{2}.
$$
\end{propo}

The following Proposition \ref{propo5} is proved in Section \ref{sec technical}.

\begin{propo}\label{propo5}
$$
\text{\rm (a)} \hspace*{28mm} F\big(\del_0 C \setminus \big[ (1-s)C + sA \big] \big) + F\big(\del_0 C \setminus \big[ (1-s)C + sB \big] \big)
\hspace*{33mm} $$ $$
- F\big(\del_0 C \setminus \big[ (1-s)C + s(A \cup B) \big] \big)
$$ $$
+ F\big( \big[ (1-s) \del^A C + sA \big] \setminus C \big) + F\big( \big[ (1-s) \del^B C + sB \big] \setminus C \big)
$$ $$
= -a_2 s -2b_4 s + o(|s|);
$$ $$
\text{\rm (b)} \quad F\big( \big[ (1-s) \del^A C + sA \big] \setminus \big[ (1-s)C + sB \big] \big) +
F\big( \big[ (1-s) \del^B C + sB \big] \setminus \big[ (1-s)C + sA \big] \big)
$$ $$
= -a_3 s + o(|s|);
$$
$$\text{\rm (c)} \ F\Big( \big[ (1-s) \del^A C + sA \big] \setminus \big( C \cup \big[ (1-s)C + sB \big] \big) \Big) +
F\Big( \big[ (1-s) \del^B C + sB \big] \setminus \big( C \cup \big[ (1-s)C + sA \big] \big) \Big) = o(|s|).$$
\end{propo}

Summing up the values obtained in Lemmas \ref{propo5.0} and \ref{propo5} and using Proposition \ref{propo3}, one comes to the statement of Lemma \ref{ls<0}.

\section{Proofs of technical statements}\label{sec technical}

\subsection{Proof of Lemma \ref{lt1}}

Assume the contrary, that is, a point $r_0 \in U$ is extreme. (It follows of course that $n_{r_0} \in \EEE$.) Take  the $\ve$-neighborhood of $r_0$ (let it be denoted by $B_\ve(r_0)$) with $\ve > 0$ sufficiently small, so as $B_\ve(r_0) \cap \pl C$ is contained in $U$. Since $r_0$ is extreme, it is not contained in {\conv}$(C \setminus B_\ve(r_0))$. Draw a plane $\Pi$ strictly separating {\conv}$(C \setminus B_\ve(r_0))$ and $r_0$.

Let $\pl_+ C$ be the part of $\pl C$ cut off by $\Pi$ and containing $r_0$. That is, $\pl_+ C$ is the intersection of $\pl C$ with the open half-space bounded by $\Pi$ and containing $r_0$. Take a point $r' \in \pl_+ C$ such that the tangent plane at it is parallel to $\Pi$. Take an open set $\OOO \subset S^2$ containing $n_{r'}$ and such that all points $r$ with the outward normal $n_r$ in $\OOO$ lie in $\pl_+ C$. Take $\check n$ in the set $\OOO \setminus \EEE$, which by the hypothesis of the lemma is not empty.

The planar set $\{ r : n_r = \check n \}$ is convex, maybe degenerating to a line segment or a point. Let $\check r$ be an extreme point of it (if the set is a segment or a point then $\check r$ coincides with an endpoint of the segment or with that point, respectively). Then $\check{r}$ is an extreme point of $C$,\, ${\check r} \in U$, and $n_{\check r} = \check n \not\in \EEE$, in contradiction with the hypothesis of Lemma \ref{lt1}.

\subsection{Proof of Lemma \ref{lt2}}\label{sub lt2}

For $0 \le s \le 1$ we have $C \subset C_s \subset \tilde C$, hence $C_s$ is the epigraph of a convex function $u^{(s)}$ defined on $\Om$ and satisfying $\tilde u \le u^{(s)} \le u$. Since the function $\tilde u$ satisfies statements (a) and (b), so does the function $u^{(s)}$.

Let now $s < 0$. The point $A' = A + (0, 0, \ve)$ is contained in $C$. The set $(1-s) C + s A'$ is the homothety of $C$ with the center $A' \in C$ and ratio $1-s > 1$, and therefore, contains $C$. Hence the set $(1-s)C + sA = (1-s)C + sA' + (0, 0, -s\ve)$ contains $C + (0, 0, -s\ve)$. The same argument holds for $(1-s)C + sB$. It follows that $C + (0, 0, -s\ve) \subset C_s \subset C$, and therefore, $C_s$ is the epigraph of a convex function $u^{(s)}$ defined on $\Om$ and satisfying $u \le u^{(s)} \le u - s\ve$. For $-1 < s < 0$, $u \le u^{(s)} < u + \ve$, and so, statement (b) is true.

For any $x \in \Om \setminus \UUU$, the line segment joining $(x, u(x))$ and $A$ is the disjoint union of two non-degenerate segments. One of them, let it be $\lam_1(x)$, is a closed segment with one endpoint at $(x, u(x))$, and is contained in $C$; let its length be $|\lam_1(x)| = l_1(x)$. The other one is a semiopen segment with one endpoint at $A$, and is disjoint with $C$; let its length be $l_2(x)$. The function $l_1(x)$ is continuous, positive, and defined on the compact set $\Om \setminus \UUU$; hence it is bounded below by a positive constant $c_1$. The function $l_2(x)$ is bounded above by a constant $c_2$.

Let $-c_1/c_2 < s < 0$; then for all $x \in \Om \setminus \UUU$, the point $(x, u(x))$ is contained in the segment $(1-s) \lam_1(x) + s A$, and therefore, in $(1-s)C + sA$. A similar argument is valid also for $(1-s)C + sB$; hence choosing $s_0 = -c_1/c_2$, we have the following: for $s_0 < s < 0$ and for all $x \in \Om \setminus \UUU$, the point $(x, u(x))$ is contained in $[(1-s)C + sA] \cap [(1-s)C + sB]$. Since by definition $(x, u(x)) \in \pl C$, one concludes that this point belongs to $\pl C_s$, and therefore, $u(x) = u^{(s)}(x)$. We have proved that for $s_0 < s < 0$, $u\rfloor_{\Om\setminus\UUU} = u^{(s)}\rfloor_{\Om\setminus\UUU}$, and so, statement (a) is true.

\subsection{Lemmas \ref{labc} and \ref{l convex}}

Later on we will need the following technical lemmas.

\begin{lemma}\label{labc}
(a) If a point $P$ belongs to $\del^A C$ then the interval $(A,\, P)$ does not intersect $C$.

(b) If an interval $(A,\, P)$ intersects $\pl_0 C \cup \pl_4 C$ then $P \not\in C$.

(c) If an interval $(A,\, P)$ intersects $\pl_2 C$ then either $P \not\in C$, or $P \in \pl_2 C$.

  (d) If both segments $I_+$ and $I_-$ degenerate to points, then $\overline{\pl_2^A C} \cap \overline{\pl_2^B C} = I_+ \cup I_-$. If only one of them, say $I_+$, is a point, then $\overline{\pl_2^A C} \cap \overline{\pl_2^B C} = I_+$. If both segments are non-degenerate, then $\overline{\pl_2^A C} \cap \overline{\pl_2^B C} = \emptyset$.
\end{lemma}

\begin{proof}
(a) The point $P$ lies on $\Pi^n$ with a certain $n \in \AAA_1^A \cup \AAA_3$. The corresponding plane $\tilde\Pi^n$ contains $A$ and does not intersect $C$. It follows that $\Pi^n$ separates $C$ and $A$ and does not contain $A$, and therefore, the interval $(A,\, P)$ does not intersect $C$.

(b) Each point of intersection of $(A,\, P)$ and $\pl_0 C$ lies on the plane $\Pi^n = \tilde\Pi^n$, with an $n \in \AAA_0$. The point $A$ does not belong to $\Pi^n$ and lies in the same closed half-space bounded by $\Pi^n$ as $C$. It follows that $P$ lies in the complement to this half-space, and therefore, is not contained in $C$.

Similarly, each point of intersection of $(A,\, P)$ and $\pl_4 C$ lies on the plane $\Pi^n = \tilde\Pi^n$, with an $n = n_\pm \in \AAA_4$. The intersection of $\Pi^n$ and $C$ is the segment $I_\pm$ parallel to $I$ and non-parallel to to $(A,\, P)$. Since $(A,\, P)$ intersects the segment $I_\pm$, the point $P$ does not belong to this segment, and therefore, does not belong to $C$.

(c) Let $(A,\, P)$ intersect $\pl_2^B C$; then each point of intersection lies on $\Pi^n = \tilde\Pi^n$, with an $n \in \AAA_2^B$. Hence, like in the previous item, $A$ does not belong to $\Pi^n$ and lies in the same closed half-space bounded by $\Pi^n$ as $C$, and thus, $P$ does not belong to $C$.

Let now $(A,\, P)$ intersect $\pl_2^A C$. Then each point of intersection lies on a certain plane $\Pi^n = \tilde\Pi^n$ containing $A$, and therefore, $P$ is also contained in this plane. Thus, if $P \in C$ then $\Pi^n$ is a plane of support to $C$ at $P$, and so, $P \in \pl_2 C$.

(d) Let $I_+$ and $I_-$ be the segments $[A_+,\, B_+ ]$ and $[A_-,\, B_-]$, respectively, parallel and co-directional with $[A,\, B]$. Recall that $\Pi[x_1]$ is the half-plane bounded by line $AB$ and containing the line $(x_1, \RRR, w(x_1))$. The intersection of each half-plane $\Pi[x_1]$,\, $x_1^- < x_1 <x_1^+$ with $\pl_2^A C$ is the segment $[A'(x_1),\, A''(x_1)]$, and with $\pl_2^B C$, the segment $[B'(x_1),\, B''(x_1)]$. These segments are disjoint.
The half-planes $\Pi[x_1^-]$ and $\Pi[x_1^+]$ contain, respectively, the segments $I_-$ and $I_+$, and do not intersect the sets $\pl_2^A C$ and $\pl_2^B C$.

For all $x_1^- < \al < \bt <x_1^+$, the distance between the part of $\pl_2^A C$ contained between the half-planes $\Pi[\al]$ and $\Pi[\bt]$ and the part of $\pl_2^B C$ contained between them is positive. Any family of points $a(x_1) \in [A'(x_1),\, A''(x_1)]$,\, $x_1^- < x_1 <x_1^+$ converges to $A_-$ as $x_1 \to x_1^-$ and to $A_+$ as $x_1 \to x_1^+$. Similarly, a family of points $b(x_1) \in [B'(x_1),\, B''(x_1)]$ converges to $B_-$ as $x_1 \to x_1^-$ and to $B_+$ as $x_1 \to x_1^+$.

It follows that the closure of $\pl_2^A C$ contains the points $A_-$ and $A_+$, and does not contain the remaining points of $I_-$ and $I_+$. Similarly, the closure of $\pl_2^B C$ contains the points $B_-$ and $B_+$, and does not contain the remaining points of $I_-$ and $I_+$. Claim (d) follows from this.
\end{proof}

\begin{lemma}\label{l convex}
Let $C_1$ and $C_2$ be two convex sets with nonempty interior in $\RRR^3$, and let $B_1 \subset \pl C_1$ and $B_2 \subset \pl C_2$ be two Borel sets such that the sets of normals to $C_1$ at the points of $B_1$ and to $C_2$ at the points of $B_2$ are disjoint. Then $B_1 \cap B_2$ has zero 2-dimensional Hausdorff measure.
\end{lemma}

\begin{proof}
We have
$B_1 \cap B_2 \subset \pl C_1 \cap \pl C_2 \subset \pl(C_1 \cap C_2).$ Each point $r \in B_1 \cap B_2$ is a singular point of $\pl(C_1 \cap C_2).$ Indeed, if $r$ is a singular point of $\pl C_1$ or $\pl C_2$, this is true. If, otherwise, it is a regular point of both $\pl C_1$ and $\pl C_2$, and $n_1$ and $n_2$ are the corresponding outward normals, then according to the hypothesis of the lemma $n_1 \ne n_2$, and therefore, $r$ is singular. The set of singular points of $\pl(C_1 \cap C_2)$ has zero 2-dimensional Hausdorff measure.
\end{proof}

\subsection{Proof of Proposition \ref{propo3}}

For $s < 0$,\, $\pl C_s$ is the disjoint union of three sets,
$$
\pl C \cap \big[ (1-s)C + sA \big] \cap \big[ (1-s)C + sB \big],
$$ $$
C \cap \pl\big[ (1-s)C + sA \big] \cap \big[ (1-s)C + sB \big],
$$ $$
C \cap \big[ (1-s)C + sA \big] \cap \pl\big[ (1-s)C + sB \big],
$$
which contain the sets $G_0^s,\, G_1^s,\, G_2^s$, respectively. Therefore  the union of the sets $G_0^s,\, G_1^s,\, G_2^s$ is contained in $\pl C_s$. It remains to prove the reverse inclusion.

The boundary $\pl C$ is the disjoint union $\pl C = \del_0 C \sqcup \del^A C \sqcup \pl^B C$.
 Additionally, $\del^A C \cap \big[ (1-s)C + sA \big] = \emptyset$
 (otherwise the open segment joining $A$ and a point of $\del^A C$ would contain a point of $C$, which contradicts statement (a) of Lemma \ref{labc}),
 and similarly, $\del^B C \cap \big[ (1-s)C + sB \big] = \emptyset$.
It follows that
\beq\label{total0}
\pl C \cap \big[ (1-s)C + sA \big] \cap \big[ (1-s)C + sB \big] = \del_0 C \cap \big[ (1-s)C + sA \big] \cap \big[ (1-s)C + sB \big] = G_0^s.
\eeq

We have $\pl\big[ (1-s)C + sA \big] =  (1-s) \pl C + sA$. Let $x \in [(1-s) \del_0 C + sA] \cap C$; then $x = (1-s) x_1 + sA \in C$ for some $x_1 \in \del_0 C$.
It follows from statements (b) and (c) of Lemma \ref{labc} that both $x$ and $x_1$ lie in $\pl_2 C$,
and therefore, $x \in \del_0 C \cap \big[ (1-s)C + sA \big]$. Hence
\beq\label{star}
[(1-s) \del_0 C + sA] \cap C \subset \del_0 C \cap \big[ (1-s)C + sA \big].
\eeq
Similarly, replacing $A$ with $B$ in \eqref{star}, one gets
\beq\label{starB}
[(1-s) \del_0 C + sB] \cap C \subset \del_0 C \cap \big[ (1-s)C + sB \big].
\eeq

Let us now prove that
\beq\label{star2}
\big[ (1-s) \del ^B C + sA \big] \cap \big[ (1-s)C + sB \big] = \emptyset.
\eeq
Indeed, otherwise there exist two points $x_1 \in \del^B C$ and $x_2 \in C$ such that $(1-s) x_1 + sA = (1-s) x_2 + sB$. Let the outward normal to the plane tangent to $C$ at $x_1$ be denoted by $n_{x_1}$. The plane of support to $\tilde C$ with the outward normal $n_{x_1}$ contains $B$ and does not contain $A$, hence the vectors $n_{x_1}$ and $\overrightarrow{BA}$ form an obtuse angle, that is, $\langle n_{x_1},\, A-B \rangle < 0$. Since $(1-s) (x_2 - x_1) = s(A-B)$ and $s<0$, we have $\langle n_{x_1},\, x_2 - x_1 \rangle > 0$. This means that $x_2$ lies in the open subspace bounded by the tangent plane to $C$ at $x_1$ that does not contain $C$, that is, $x_2 \not\in C$, in contradiction with our assumption.

From \eqref{star} and \eqref{star2} we obtain
\beqo\label{ttttt}
\begin{split}
C \cap \pl \big[ (1-s) C + sA \big] \cap \big[ (1-s)C + sB \big] &= \Big( C \cap \big[ (1-s) \del_0 C + sA \big] \cap \big[ (1-s)C + sB \big] \Big) \cup \\
\Big( C \cap \big[ (1-s) \del^B C + sA \big] \cap \big[ (1-s)C + sB \big] \Big) &\cup \Big( C \cap \big[ (1-s) \del^A C + sA \big] \cap \big[ (1-s)C + sB \big] \Big) \subset \\
\Big( \del_0 C \cap \big[ (1-s)C + sA \big] \cap \big[ (1-s)C + sB \big] \Big) &\cup \Big( C \cap \big[ (1-s) \del^A C + sA \big] \cap \big[ (1-s)C + sB \big] \Big)
\end{split}
\eeqo
\beq\label{total1}
= G_0^s \cup G_1^s.
\eeq

Further, we have $\pl\big[ (1-s)C + sB \big] =  (1-s) \pl C + sB$. Let us prove that
\beq\label{star4}
\big[ (1-s)C + sA \big] \cap \big[ (1-s) \del ^A C + sB \big] \subset \big[ (1-s) \del^A C + sA \big] \cap \big[ (1-s)C + sB \big]
\eeq
Take a point $x \in \big[ (1-s)C + sA \big] \cap \big[ (1-s) \del ^A C + sB \big]$; there exist $x_1 \in \del^A C$ and $x_2 \in C$ such that $x = (1-s) x_1 + sB = (1-s) x_2 + sA$. The plane of support to $\tilde C$ with the outward normal $n_{x_1}$ contains $A$ (and may contain or not contain $B$), hence $\langle n_{x_1},\, B-A \rangle \le 0$, and using that $(1-s) (x_2 - x_1) = s(B-A)$, we get $\langle n_{x_1},\, x_2 - x_1 \rangle \ge 0$. Since $x_2 \in C$, we conclude that $\langle n_{x_1},\, x_2 - x_1 \rangle = 0$. Hence the tangent planes to $C$ at $x_1$ and $x_2$ coincide, and the plane of support to $\tilde C$ with the same outward normal contains $I$. It follows that $x_1$ lies in $\pl_3 C$, and therefore, $x_2$ also lies in $\pl_3 C$, and $x \in \big[ (1-s) \pl_3 C + sA \big] \cap \big[ (1-s) \pl_3 C + sB \big] \subset \big[ (1-s) \del^A C + sA \big] \cap \big[ (1-s)C + sB \big]$. Formula \eqref{star4} is proved.

Thus, using \eqref{starB} and \eqref{star4}, one can write
\beqo\label{totttt}
\begin{split}
\big[ (1-s)C + sA \big] \cap \pl \big[ (1-s) C + sB \big] \cap C &= \Big( \big[ (1-s)C + sA \big] \cap \big[ (1-s) \del_0 C + sB \big] \cap C \Big) \cup\\
\Big( \big[ (1-s)C + sA \big] \cap \big[ (1-s) \del^A C + sB \big] \cap C \Big) &\cup \Big( \big[ (1-s)C + sA \big] \cap \big[ (1-s) \del^B C + sB \big] \cap C \Big) \subset \\
\Big( \big[ (1-s)C + sA \big] \cap \del_0 C \cap \big[ (1-s)C + sB \big] \Big) &\cup \Big( \big[ (1-s)C + sA \big] \cap \big[ (1-s) \del^B C + sB \big] \cap C \Big)
\end{split}
\eeqo
\beq\label{total2}
\cup \Big( \big[ (1-s) \del^A C + sA \big] \cap \big[ (1-s)C + sB \big] \cap C
\Big) = G_0^s \cup G_1^s \cup G_2^s.
\eeq

It follows from \eqref{total0}, \eqref{total1}, and \eqref{total2} that $\pl C_s \subset G_0^s \cup G_1^s \cup G_2^s$.

Of course, $F(u^{(s)}) = F(\pl C_s)$. In order to prove that $F(\pl C_s) = F(G_0^s) + F(G_1^s) + F(G_2^s)$, it suffices to show that the intersections $G_0^s \cap G_1^s$,\, $G_1^s \cap G_2^s$,\, $G_0^s \cap G_2^s$ have zero measure. We have
$$
G_0^s \cap G_1^s \subset \del_0 C \cap \big[ (1-s) \del^AC + sA \big],
$$ $$
G_1^s \cap G_2^s \subset \big[ (1-s) \del^AC + sA \big] \cap  \big[ (1-s) \del^B C + sB \big],
$$  $$
G_0^s \cap G_2^s \subset \del_0 C \cap \big[ (1-s) \del^B C + sB \big].
$$
The sets of outward normals at the points of the sets $\del_0 C \subset \pl C$,\, $(1-s) \del^AC + sA \subset \pl [(1-s) C + sA]$, and $(1-s) \del^B C + sB \subset \pl [(1-s) C + sB]$ are pairwise disjoint, hence by Lemma \ref{l convex}, the pairwise intersections of these sets have zero measure. This finishes the proof of Proposition \ref{propo3}.

\subsection{Proof of Proposition \ref{propo5}}

Denote by $\Pi_\vphi$ the half-plane bounded by the vertical line through $A$ and making the angle $\vphi$ with the ray $AB$, where the angle $\vphi \in [0,\, 2\pi)$ is counted, say, counterclockwise from $AB$. In each half-plane $\Pi_\vphi$ draw the ray with the vertex at $A$ touching the planar set $\Pi_\vphi \cap C$. The intersection of the ray with this set is a segment (maybe degenerating to a point); let it be $[A'_\vphi,\, A''_\vphi]$ (see Fig.~\ref{fig vphi}). Correspondingly, the semiopen segment $[A,\, A'_\vphi)$ does not intersect $C$.

The segment $[A'_\vphi,\, A''_\vphi]$ does not intersect the set $\pl_0 C \cup \pl_1 ^A C \cup \pl_2^B C \cup \pl_3 C$. Indeed, assume the contrary: a point of this set belongs to $[A'_\vphi,\, A''_\vphi]$. The line $AA'_\vphi$ is contained in a plane of support to $C$ at that point. However, any plane of support to $C$ at a point of $\pl_0 C \cup \pl_1 ^A C \cup \pl_2^B C \cup \pl_3 C$ is a plane $\Pi^n$ with $n \in \AAA_0 \cup \AAA_1^A \cup \AAA_2^B \cup \AAA_3$. No plane of this form contains $A$.

Thus, there may be three cases.
\vspace{1mm}

(i) The segment $[A'_\vphi,\, A''_\vphi]$ intersects $\pl_2^A C$; then it belongs to $\pl_2^A C$. The set of corresponding values of $\vphi$ is denoted by $\Phi_2^A$.

(ii) The segment intersects $\pl_4 C$; then it degenerates to a point that lies in $\pl_4 C$. (Recall that $\pl_4 C$ is the union of two segments, $I_-$ and $I_+$, parallel to $AB$.) The set of corresponding values of $\vphi$ is denoted by $\Phi_4^A.$

(iii) The segment intersects $\pl_1 C$; then it lies in $\pl_1 C$ (and therefore, in $\pl_1^B C$). The set of corresponding values of $\vphi$ is denoted by $\Phi_{1}^A.$
\vspace{1mm}

The sets $\Phi_2^A$, $\Phi_4^A$, and $\Phi_{1}^A$ are disjoint, and their union is $[0,\, 2\pi)$.

Let us prove the following auxiliary lemma.

\begin{lemma}\label{l intersect}
(a) If $\vphi \in \Phi_2^A \cup \Phi_4^A$ then an arc of the curve $\pl C \cap \Pi_\vphi$ bounded below by the point $A'_\vphi$ and containing this point lies in $\del_0 C$.

(b) If $\vphi \in \Phi_{1}^A$ then an arc of the curve $\pl C \cap \Pi_\vphi$ containing the point $A'_\vphi$ in its interior is disjoint with $\del_0 C$.

(c) If $\vphi \in \Phi_2^A$ then an open arc of the curve $\pl C \cap \Pi_\vphi$ bounded above by $A'_\vphi$ lies in $\pl_1^A C$.

(d) If $\vphi \in \Phi_1^A$ then an arc of the curve $\pl C \cap \Pi_\vphi$ containing $A'_\vphi$ in its interior is disjoint with $\pl_1^A C$.
\end{lemma}

\begin{proof}
(a) The segment $[A'_\vphi,\, A''_\vphi]$ is contained in $\pl_2^A C$ or in $\pl_4 C$ (in the latter case it degenerates to a point), and therefore is contained in $\del_0 C$; therefore it suffices to consider the part of the curve above $A''_\vphi$.

Let $n = n_{A''_\vphi}$ be the outward normal to $C$ at $A''_\vphi$. Since $n \in \AAA_2^A$, we have $\langle \overrightarrow{A_\vphi'' B},\, n \rangle < 0$. This inequality remains valid when $A_\vphi''$ is substituted with $Q$ and $n$ is substituted with $n_Q$, where $Q$ is a point of a sufficiently small arc of the curve $\pl C \cap \Pi_\vphi$ bounded below by $A''_\vphi$. Additionally, the segment $AQ$ intersects the interior of $C$. It follows that $Q \in \pl_0 C \subset \del_0 C$, and so, claim (a) is proved.

(b) The point $A'_\vphi$ belongs to $\pl_1 C$, and therefore, lies in the interior of $\tilde C$. Hence an arc of the curve $\pl C \cap \Pi_\vphi$ containing $A'_\vphi$ in its interior also lies in the interior of $\tilde C$, and so, is disjoint with $\del_0 C$. Claim (b) is proved.

(c) Draw the tangent plane to $C$ at $A_\vphi'$. The point $B$ belongs to the interior of the half-space bounded by this plane and containing $C$. Hence $B$ also belongs to the interior of the half-space bounded by the tangent plane at any point $P$ of a sufficiently small arc of $\pl C \cap \Pi_\vphi$ bounded above by $A'_\vphi$. Additionally, the point $A$ belongs to the interior of the complementary half-space. It follows that the plane through $A$ parallel to the tangent plane at $P$ is supporting for $\tilde C$, and its intersection with $\tilde C$ is $A$. Hence the point $P$ belongs to $\pl_1^A C$, and claim (c) is proved.

(d)  The point $A'_\vphi$ belongs to $\pl_1^B C$, hence so does a neighborhood of this point in $\pl C$. The arc of $\pl C \cap \Pi_\vphi$ contained in this neighborhood does not intersect $\pl_1^A C$. Claim (d) is proved.

\end{proof}

Denote by $\mathbf{\Pi}_i^A$ the union of half-planes $\Pi_\vphi$ with $\vphi \in \Phi_i^A$,\, $\mathbf{\Pi}_i^A = \cup_{\vphi \in \Phi_i^A} \Pi_\vphi$,\, $i = 2,\, 4,\, 1$. We have the disjoint union
$$
\mathbf{\Pi}_2^A \sqcup \mathbf{\Pi}_4^A \sqcup \mathbf{\Pi}_1^A = \RRR^3.$$

In a similar fashion, there are defined the sets $\Phi_2^B$, $\Phi_4^B$, and $\Phi_{1}^B$ and the sets $\mathbf{\Pi}_2^B$,\, $\mathbf{\Pi}_4^B$, and $\mathbf{\Pi}_1^B$.

Fix a value $s < 0$ and define the points $P_\vphi(s)$ and $Q_\vphi(s)$ on $\pl C \cap \Pi_\vphi$ as follows. If the length of the segment $[A'_\vphi,\, A''_\vphi]$ is greater than or equal to $-s$ times the length of $[A,\, A'_\vphi]$, then set $P_\vphi = P_\vphi(s) = A'_\vphi$ and take the point $Q_\vphi = Q_\vphi(s)$ on $[A'_\vphi,\, A''_\vphi]$ such that
\beq\label{cond=}
\frac{|P_\vphi Q_\vphi|}{|AP_\vphi|} = -s.
\eeq
If, otherwise, $|A'_\vphi A''_\vphi| < -s |A A'_\vphi|$ then draw the ray in $\Pi_\vphi$ with the vertex at $A$ intersecting $C \cap \Pi_\vphi$ such that the two points of intersection $P_\vphi$ and $Q_\vphi$ satisfy condition \eqref{cond=}. This ray is unique for all $\vphi$. See Fig.~\ref{fig vphi}.

                                       \begin{figure}[h]
\centering
\includegraphics[scale=0.25]{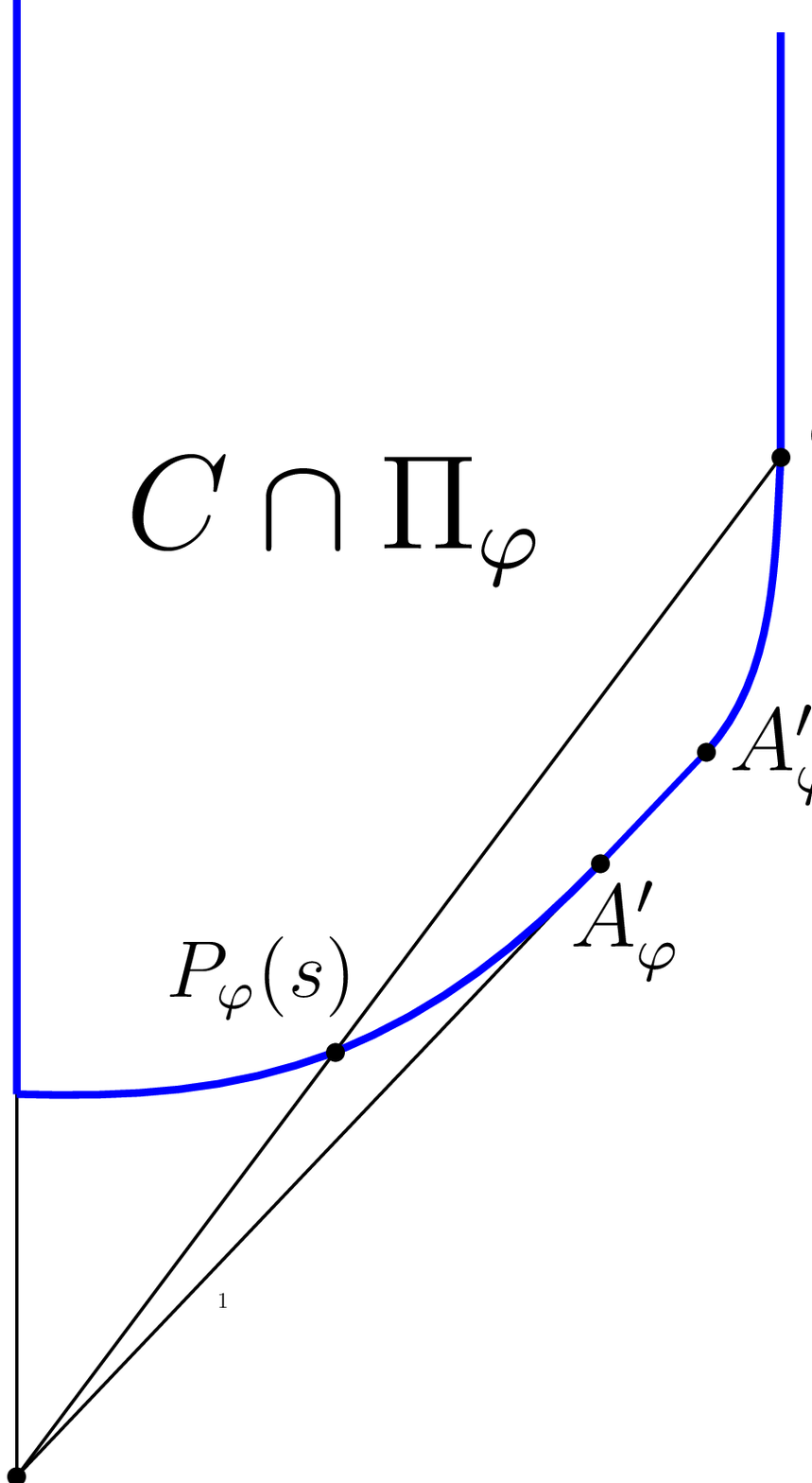}
\caption{The section of $C$ by the half-plane $\Pi_\vphi$.}
\label{fig vphi}
\end{figure}

For any $\vphi$ consider the open arc $(\stackrel{\frown}{P_\vphi, Q_\vphi})$, that is, the part of $\pl C \cap \Pi_\vphi$ that is bounded by the points $P_\vphi$ and $Q_\vphi$ and does not contain them. Any such arc  is divided by the point $A'_\vphi$ into two parts, the lower one $(\stackrel{\frown}{P_\vphi, A'_\vphi})$ (which may be empty) and the upper one $[\stackrel{\frown}{A'_\vphi, Q_\vphi})$. Denote the union of the arcs $(\stackrel{\frown}{P_\vphi, Q_\vphi})$ by $\Upsilon^A(s)$, the union of the lower arcs $(\stackrel{\frown}{P_\vphi, A'_\vphi})$ by $\LLL^A(s)$, and the union of the upper arcs $[\stackrel{\frown}{A'_\vphi, Q_\vphi})$ by $\UUU^A(s)$. The following disjoint union takes place, $\Upsilon^A(s) = \LLL^A(s) \sqcup \UUU^A(s)$.

Let us also designate by $\Upsilon_i^A(s)$, $\LLL_i^A(s)$, $\UUU_i^A(s)$ the intersections of $\Upsilon^A(s)$, $\LLL^A(s)$, $\UUU^A(s)$, respectively, with $\mathbf{\Phi}_i^A$, with $i = 2,\, 4,\, 1$. In other words, we have
$$
\Upsilon^A(s) = \cup_{0 \le \vphi < 2\pi} (\stackrel{\frown}{P_\vphi, Q_\vphi}), \quad
\LLL^A(s) = \cup_{0 \le \vphi < 2\pi} (\stackrel{\frown}{P_\vphi, A'_\vphi}), \quad
\UUU^A(s) = \cup_{0 \le \vphi < 2\pi} [\stackrel{\frown}{A'_\vphi, Q_\vphi}), $$
$$
\Upsilon_i^A(s) = \cup_{\vphi \in \Phi_i^A} (\stackrel{\frown}{P_\vphi, Q_\vphi}), \ \
\LLL_i^A(s) = \cup_{\vphi \in \Phi_i^A} (\stackrel{\frown}{P_\vphi, A'_\vphi}), \ \
\UUU_i^A(s) = \cup_{\vphi \in \Phi_i^A} [\stackrel{\frown}{A'_\vphi, Q_\vphi}), \ \, i = 2,\, 4,\, 1.$$
In a similar way we define the sets $\Upsilon_i^B(s)$, $\LLL_i^B(s)$, $\UUU_i^B(s)$.
     Of course, these sets belong to $\pl C$, and both $\LLL^A(s)$ and $\LLL^B(s)$ belong to $\del^A C \cup \del^B C = \pl_1 C \cup \pl_3 C$.

From now on in this subsection, $o(1)$ means a function of $s$ and $\vphi$ that goes to 0 uniformly in $\vphi$ as $s \to 0^-$, and $o(s)$ and $O(s)$ have a similar meaning.

We will need the following lemma.

\begin{lemma}\label{l areas}
The two-dimensional Lebesgue measures of the sets $\LLL_4^A(s) \setminus \LLL_4^B(s)$,\, $\UUU_4^A(s) \setminus \UUU_4^B(s)$,\, $\LLL_4^B(s) \setminus \LLL_4^A(s)$,\, $\UUU_4^B(s) \setminus \UUU_4^A(s)$ are $o(s)$.
\end{lemma}

\begin{proof}
We will prove the lemma for the set $\UUU_4^A(s) \setminus \UUU_4^B(s)$. The proofs for the other sets are basically the same.

If both $I_+$ and $I_-$ are points, there is nothing to prove: the sets $\UUU_4^A(s)$ and $\UUU_4^B(s)$ have measure zero. If, otherwise, one segment or both are non-degenerate, we do the following. Let the segment $I_+$ be non-degenerate; then consider two smaller segments with the length $\sqrt{|s|}$ contained in $I_+$ and adjacent to its endpoints. Let $I_+'(s)$ be $I_+$ minus the two smaller segments (see Fig.~\ref{fig seg}).

                                       \begin{figure}[h]
\centering
\includegraphics[scale=0.25]{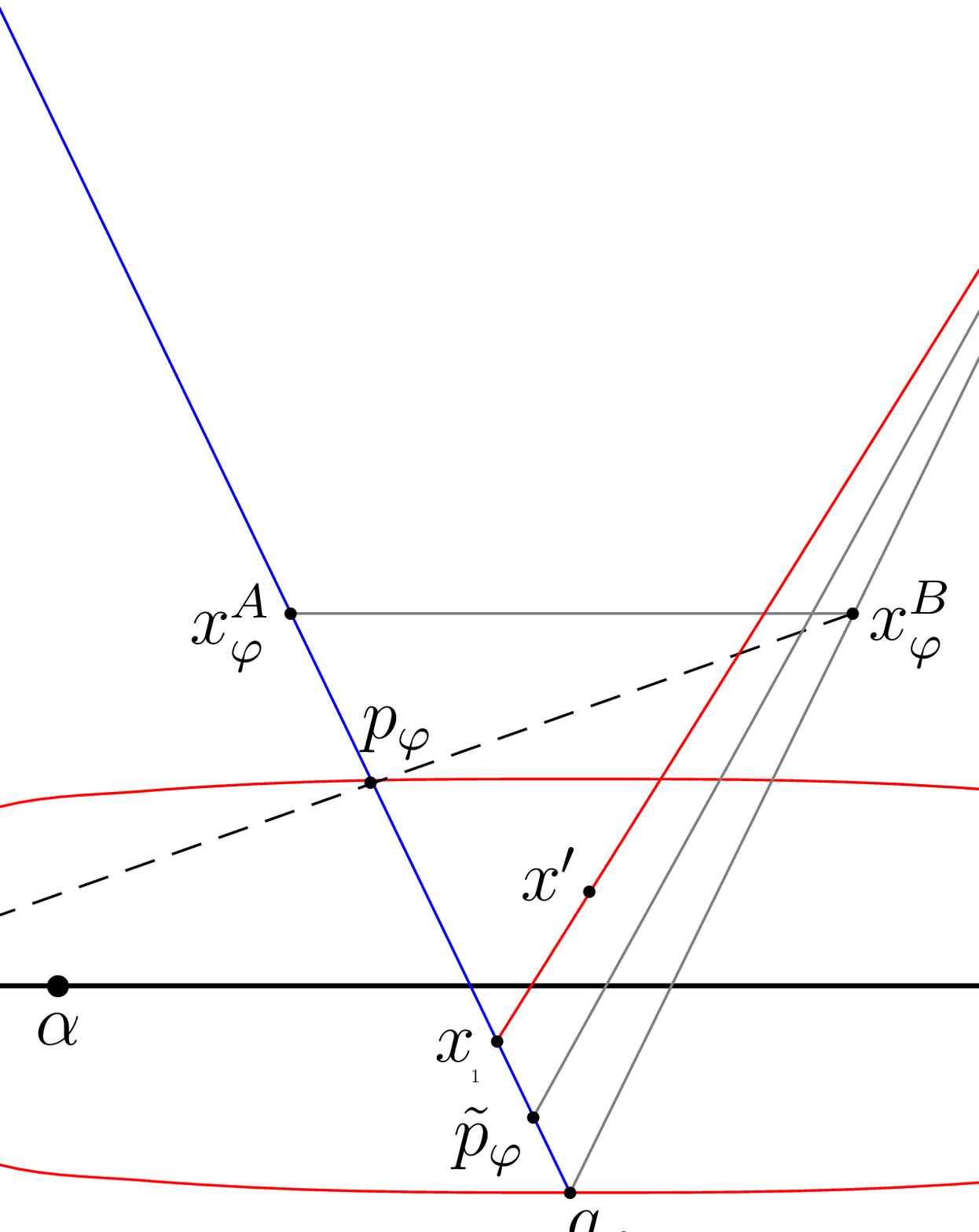}
\caption{The vertical projections of $A$, $B$, $P_\vphi$, $Q_\vphi$, $X$ on the $x$-plane are the points $x^A$, $x^B$, $p_\vphi$, $q_\vphi$, $x$ indicated in the figure. The projection of the segment $I_+$ is $x_+^A x_+^B$, and the projection of $I_+'(s)$ is $\al \bt$. 
The convex set bounded by the closed curve is the projection of the intersection of $C$ with the plane through $A$, $B$, and $P_\vphi$.}
\label{fig seg}
\end{figure}

The union of all arcs $(\stackrel{\frown}{P_\vphi, Q_\vphi})$ intersecting the two smaller segments is a subset of $\Upsilon_4^A(s)$ and has the area $O(|s|^{3/2})$. The complementary part of $\Upsilon_4^A(s)$ is the union of arcs $(\stackrel{\frown}{P_\vphi, Q_\vphi})$ that intersect $I_+'(s)$. It is divided by the segment $I_+'(s)$ into two parts. The upper part (lower in Fig.~\ref{fig seg}) belongs to $\UUU_4^A(s)$, and the lower part (upper in Fig.~\ref{fig seg}) belongs to $\LLL_4^A(s)$. It remains to prove that the set of points in the part of $\UUU_4^A(s)$ that do not belong to $\UUU_4^B(s)$ has the measure $o(s)$.

Denote by $x^A$ and $x^B$ the vertical projections of $A$ and $B$ on the $x$-plane; that is, $x^A = x^0 - (0,\del)$ and $x^B = x^0 + (0,\del)$. Denote by $x^A_+$ and $x^B_+$ the endpoints of $S_+$ (recall that $S_+$ is the projection of $I_+$); one has $x^A_+ = x^+ - (0, a^+)$ and $x^B_+ = x^+ + (0, a^+)$. Denote by $p_\vphi = p_\vphi(s)$ and $q_\vphi = q_\vphi(s)$, respectively, the projections of the points $P_\vphi$ and $Q_\vphi$ on the $x$-plane. We assume that the arc $(\stackrel{\frown}{P_\vphi, Q_\vphi})$ intersects $I_+'(s)$, and therefore, the interval $(p_\vphi, q_\vphi)$ intersects the projection of $I_+'(s)$ (which is denoted as $\al \bt$ in Fig.~\ref{fig seg}).

Let the point $x^B_\vphi = x^B_\vphi(s)$ be the intersection of lines $x^A_+ p_\vphi$ and $x^B q_\vphi$, and let $x^A_\vphi = x^A_\vphi(s)$ be the point on $x^A q_\vphi$ such that the line $x^A_\vphi x^B_\vphi$ is parallel to $x^A x^B$.

We have $|p_\vphi q_\vphi| = -s |x^A p_\vphi| = O(s)$, the slope of the line $p_\vphi x_\vphi^B$ is $O(\sqrt{|s|})$, hence $|p_\vphi x_\vphi^B| = O(s)$ and $|p_\vphi x_\vphi^A| = O(|s|^{3/2})$. Let $\tilde p_\vphi$ be the image of $x^A_\vphi$ under the homothety with the center $x^A$ and the ratio $1-s$. Taking into account that $q_\vphi$ is the image of $p_\vphi$ under the same homothety, one concludes that $|\tilde p_\vphi q_\vphi| = (1-s) |x^A_\vphi p_\vphi| = O(|s|^{3/2})$. It follows that the length of the part of the arc $(\stackrel{\frown}{P_\vphi, Q_\vphi})$ that projects onto $\tilde p_\vphi q_\vphi$ is also $O(|s|^{3/2})$, and therefore, the area of the part of $\UUU_4^A(s)$ that projects onto the union of the segments $\tilde p_\vphi q_\vphi$ over $\vphi$ is $O(|s|^{3/2}) = o(s)$.

Let us prove that each point of $\UUU_4^A(s)$ that projects to $p_\vphi \tilde p_\vphi$
also belongs to $\UUU_4^B(s)$. This will finish the proof of Lemma \ref{l areas}.


Consider the plane $\Pi$ through the lines $AB$ and $AP_\vphi$. The intersection of $\Pi$ with $C$ is a planar convex set, and the points $P_\vphi$ and $Q_\vphi$ lie on its boundary. The vertical projection of this set on the $x$-plane is bounded by the closed curve in Fig.~\ref{fig seg}.

Take a point $X$ on the arc $(\stackrel{\frown}{P_\vphi, Q_\vphi})$ that lies above the segment $I'_+(s)$ (and therefore, belongs to $\UUU_4^A(s)$) and projects to $p_\vphi \tilde p_\vphi$. Its projection is denoted as $x$, and it lies below the segment $\al\bt$ in Fig.~\ref{fig seg}.
The straight line through $B$ and $X$ intersects $\pl C$ at two points; one of them is $X$, and the other one is denoted as $X'$ (and its projection is $x'$). The points $X$ and $X'$ bound the arc $(\stackrel{\frown}{X', X})$ in the vertical plane through $B$ and $X$. The point $X$ lies below the plane $\Pi$ and therefore so does $X'$, hence their projections on the $x$-plane belong to the projection of $\Pi \cap C$ (lie inside the closed curve in Fig.~\ref{fig seg}).

The point $x'$ lies below the line $x_\vphi^A x_\vphi^B$ and $x$ lies on the segment $x^A \tilde p_\vphi$, hence
$$
\frac{|x'x|}{|x^B x'|} \le \frac{|\tilde p_\vphi x_\vphi^A|}{|x^A x_\vphi^A|} = -s, \quad \text{and therefore}, \quad \frac{|X'X|}{|B X'|} \le -s.
$$
It follows that the arc $(\stackrel{\frown}{X', X})$ belongs to another arc in the vertical plane, say $(\stackrel{\frown}{P', Q'})$, that satisfies the equation $\frac{|P' Q'|}{|BP'|} = -s$, and therefore, lies in $\Upsilon_4^B$. Thus, $X$ also lies in $\Upsilon_4^B$, and therefore, in $\UUU_4^B$.
\end{proof}

Take a point $Q$ on $\del_0 C$ and draw the segment $[A,\, Q]$. Let $[P, Q]$ be the intersection of this segment with $C$. The condition that $Q$ is contained in $\del_0 C \setminus [(1-s)C + sA]$ is
\beq\label{condition}
\frac{|PQ|}{|AP|} < -s.
\eeq
It follows that the curve $(\del_0 C \cap \Pi_\vphi) \setminus [(1-s)C + sA]$ is the intersection of the arc $[\stackrel{\frown}{A'_\vphi, Q_\vphi})$ with $\del_0 C$. By claims (a) and (b) of Lemma \ref{l intersect}, if $\vphi \in \Phi_2^A \cup \Phi_4^A$ then for $|s|$ sufficiently small this curve coincides with $[\stackrel{\frown}{A'_\vphi, Q_\vphi})$, and if $\vphi \in \Phi_1^A$ then for $|s|$ sufficiently small the curve is empty. It follows that the symmetric difference of the sets $\del_0 C \setminus [(1-s)C + sA]$ and $\UUU_2^A(s) \cup \UUU_4^A(s)$ has the measure $o(s)$. Hence
\beq\label{ast1}
F(\del_0 C \setminus [(1-s)C + sA]) = F(\UUU_2^A(s)) + F(\UUU_4^A(s)) + o(s).
\eeq
Similarly,
\beq\label{ast2}
F(\del_0 C \setminus [(1-s)C + sB]) = F(\UUU_2^B(s)) + F(\UUU_4^B(s)) + o(s).
\eeq

The set $\del_0 C \setminus [(1-s)C + s(A \cup B)]$ is the intersection of the sets $\del_0 C \setminus [(1-s)C + s A]$ and $\del_0 C \setminus [(1-s)C + s B]$, therefore
$$
F(\del_0 C \setminus [(1-s)C + s(A \cup B)]) = F(\UUU_2^A(s) \cap \UUU_2^B(s)) + F(\UUU_4^A(s) \cap \UUU_4^B(s)) + o(s).
$$
$\UUU_2^A(s)$ lies in the $s$-neighborhood of $\pl_2^A C$, and $\UUU_2^B(s)$ lies in the $s$-neighborhood of $\pl_2^B C$. By claim (d) of Lemma \ref{labc}, the closures of the sets $\pl_2^A C$ and $\pl_2^B C$ have at most two points in common. It follows that the 2-dimensional measure of $\UUU_2^A(s) \cap \UUU_2^B(s)$ is $o(s)$, hence $F(\UUU_2^A(s) \cap \UUU_2^B(s)) = o(s).$
Further, by Lemma \ref{l areas} we have $F(\UUU_4^A(s) \cap \UUU_4^B(s)) = F(\UUU_4^A(s)) + o(s) = F(\UUU_4^B(s)) + o(s)$, and therefore,
\beq\label{ast3}
F(\del_0 C \setminus [(1-s)C + s(A \cup B)]) = F(\UUU_4^A(s)) + o(s) = F(\UUU_4^B(s)) + o(s).
\eeq

The homothety with ratio $1-s$ and the center at $A$ sends each set $\MMM$ to $(1-s) \MMM + s A$. Its inversion is the homothety with ratio $\frac{1}{1-s}$ and the same center, and sends $\MMM$ to $\frac{1}{1-s}\, \MMM + \frac{-s}{1-s} A$. The image of the set $[ (1-s) \del^A C + sA ] \setminus C$ under the inverse homothety is $\del^A C \setminus \big[ \frac{1}{1-s}\, C + \frac{-s}{1-s} A \big]$, and
$$F \big( [ (1-s) \del^A C + sA ] \setminus C \big) =
(1-s)^2 F \big( \del^A C \cap \setminus \big[ \frac{1}{1-s}\, C + \frac{-s}{1-s} A \big] \big).
$$
It follows that
\beq\label{aprox}
\begin{split}
F \big( \del_0 C \setminus [(1-s)C + sA] \big) + F \big( [ (1-s) \del^A C + sA ] \setminus C \big) \quad & \text{equals} \\
\hspace*{-3mm}F \big( \del_0 C \setminus [(1-s)C + sA] \big) + F \big( \del^A C \setminus \big[ \frac{1}{1-s}\, C + \frac{-s}{1-s} A \big] \big) &
\ \, \text{times a factor}\ 1 + O(s).
\end{split}
\eeq
More precisely, this factor is between 1 and $(1-s)^2$. A similar estimate holds when $A$ is substituted with $B$ in \eqref{aprox}.

Relation \eqref{aprox} is useful, since the expression in its second line is easier to estimate than the expression in its first line.

Take a point $P$ from $\del^A C$ and draw the ray $AP$. The intersection of this ray with $C$ is a (non-degenerate) segment $[P, Q]$. The condition that $P$ is contained in $\del^A C \setminus \big[ \frac{1}{1-s}\, C + \frac{-s}{1-s} A \big]$ coincides with \eqref{condition}.

It follows that the curve $(\del^A C \cap \Pi_\vphi) \setminus \big[ \frac{1}{1-s}\, C + \frac{-s}{1-s} A \big]$ is the intersection of the arc $(\stackrel{\frown}{P_\vphi, A'_\vphi})$ with $\del^A C$. By claims (c) and (d) of Lemma \ref{l intersect}, if $\vphi \in \Phi_2^A$ then for $|s|$ sufficiently small the arc $(\stackrel{\frown}{P_\vphi, A'_\vphi})$ belongs to $\pl_1^A C \subset \del^A C$, and if $\vphi \in \Phi_1^A$ then for $|s|$ sufficiently small the arc is disjoint with $\pl_1^A C$, and therefore, is disjoint with  $\pl_1^A C \cup \pl_3 C = \del^A C$.

Hence the symmetric difference of the sets $\del^A C \setminus \big[ \frac{1}{1-s}\, C + \frac{-s}{1-s} A \big]$ and
$\LLL_2^A(s) \cup \big(\LLL_4^A(s) \cap \del^A C \big)$ has the 2-dimensional measure $o(s)$, and therefore,
\beq\label{ast4}
F \big( \del^A C \setminus \big[ \frac{1}{1-s}\, C + \frac{-s}{1-s} A \big] \big) = F(\LLL_2^A(s)) + F(\LLL_4^A(s) \cap \del^A C) + o(s).
\eeq
Similarly,
\beq\label{ast5}
F \big( \del^B C \setminus \big[ \frac{1}{1-s}\, C + \frac{-s}{1-s} B \big] \big) = F(\LLL_2^B(s)) + F(\LLL_4^B(s) \cap \del^B C) + o(s).
\eeq

By Lemma \ref{l areas}, the set $(\del^A C \cap \LLL_4^A(s)) \cup (\del^B C \cap \LLL_4^B(s))$ equals $(\del^A C \cup \del^B C) \cap \LLL_4^A(s)$ plus a set with 2-dimensional Lebesgue measure $o(s)$, and $\LLL_4^A(s)$ is contained in $\del^A C \cup \del^B C = \pl_1 C \cup \pl_3 C$, hence
\beq\label{ast6}
F(\del^A C \cap \LLL_4^A(s)) + F(\del^B C \cap \LLL_4^B(s)) = F(\LLL_4^A(s)) + o(s).
\eeq

Using \eqref{ast1}--\eqref{ast3} and \eqref{ast4}--\eqref{ast6}, one obtains
$$
F\big(\del_0 C \setminus \big[ (1-s)C + sA \big] \big) + F\big( \del_0 C \setminus \big[ (1-s)C + sB \big] \big)
- F\big(\del_0 C \setminus \big[ (1-s)C + s(A \cup B) \big] \big) $$ $$
+ F \big( \del^A C \setminus \big[ \frac{1}{1-s}\, C + \frac{-s}{1-s} A \big] \big) + F \big( \del^B C \setminus \big[ \frac{1}{1-s}\, C + \frac{-s}{1-s} B \big] \big)
$$ $$
= F(\UUU_2^A(s)) + F(\UUU_4^A(s)) + F(\UUU_2^B(s)) + F(\UUU_4^B(s)) -  F(\UUU_4^B(s))
$$ $$
+ F(\LLL_2^A(s)) + F(\LLL_2^B(s)) + F(\LLL_4^A(s)) + o(s) $$
     \beq\label{predv}
      = F(\Upsilon_2^A(s)) + F(\Upsilon_2^B(s)) + F(\Upsilon_4^A(s)) + o(s).
     \eeq

Recall that the vertical projection of the points $A$, $P_\vphi$, and $Q_\vphi$ on the $x$-plane are, respectively, $x^A = x^0 - (0,\del)$, $p_\vphi = p_\vphi(s)$, and $q_\vphi = q_\vphi(s)$. Denote by $x^\vphi$ the projection of the point $A'_\vphi$, and let $d_\vphi$ be the length of the segment $[x^A,\, x^\vphi]$ and $l_\vphi = l_\vphi(s)$ be the length of $[x^A,\, p_\vphi]$. Note that the points $x^A,\, p_\vphi,\, x^\vphi,\, q_\vphi$ are collinear and follow in the indicated order, besides $|p_\vphi\, q_\vphi| = -s |x^A\, p_\vphi|$ and $|x^A\, x^\vphi| = |x^A\, p_\vphi| (1 + O(s))$, that is,
$$l_\vphi = d_\vphi + O(s).$$

The surface $\pl_2^A \tilde C \setminus \pl_2^A C$ is the union of segments $[A,\, A'_\vphi]$; it is the graph of $\tilde u$ restricted on the union of segments $[x^A,\, x^\vphi]$, $\vphi \in \Phi_2^A$. The gradient of $\tilde u$ at each point of the segment $[x^A,\, x^\vphi]$ equals $\nabla u(x^\vphi)$. Hence the resistance of $\pl_2^A \tilde C \setminus \pl_2^A C$ equals
$$
F(\pl_2^A \tilde C) - F(\pl_2^A C) = \int_{\Phi_2^A} d\vphi \int_0^{d_\vphi} f(\nabla u(x^\vphi))\, r\, dr = \int_{\Phi_2^A} f(\nabla u(x^\vphi))\, \frac{d_\vphi^2}{2}\, d\vphi.
$$

Additionally, the gradient of $u$ at each point of $\Upsilon_2^A(s)$ lying in $(\stackrel{\frown}{P_\vphi, Q_\vphi})$ equals $\nabla u(x^\vphi) + o(1)$. It follows that $F(\Upsilon_2^A(s))$ equals the integral of $f(\nabla u(x^\vphi)) + o(1)$ over the union of segments $\cup_{\Phi_2^A} [p_\vphi,\, q_\vphi]$. This set is the set-theoretic difference of the union of segments $\cup_{\Phi_2^A} [x^A,\, q_\vphi]$ and the union of segments $\cup_{\Phi_2^A} [x^A,\, p_\vphi]$, the former one being the image of the latter one with ratio $1-s$ and the center at $x^A$.

It follows that
$$
F(\Upsilon_2^A(s)) = \int_{\Phi_2^A} d\vphi \int_{l_\vphi}^{(1-s)l_\vphi} [f(\nabla u(x^\vphi)) + o(1)]\, r\, dr $$ $$
= [(1-s)^2 - 1] \Big( \int_{\Phi_2^A} d\vphi \int_0^{d_\vphi + O(s)} f(\nabla u(x^\vphi))\, r\, dr + o(1) \Big)
$$ $$
= -2s \big( F(\pl_2^A \tilde C) - F(\pl_2^A C) \big) + o(s).
$$
Similarly,
$$ F(\Upsilon_2^B(s)) = -2s \big( F(\pl_2^B \tilde C) - F(\pl_2^B C) \big) + o(s),
$$
hence
\beq\label{Upsilon2}
F(\Upsilon_2^A(s)) + F(\Upsilon_2^B(s)) = -2s \big( F(\pl_2 \tilde C) - F(\pl_2 C) \big) + o(s) = -a_2 s + o(s).
\eeq

Let us now consider the value $F(\Upsilon_4^A(s))$. The vertical projection of the union of two planar triangles with the common vertex $A$  and with the bases $I_-$ and $I_+$ on the $x$-plane is again the union of two triangles with the lengths of the bases $2a_-$ and $2a_+$ and with the heights $x_1^0 - x_1^-$ and $x_1^+ - x_1^0$, and their areas are $(x_1^0 - x_1^-) a_-$ and $(x_1^+ - x_1^0) a_+$. The sum of the areas can be represented in another way as
$$
\int_{\Phi_4^A} d\vphi \int_0^{d_\vphi} r\, dr = \int_{\Phi_4^A} \frac{d_\vphi^2}{2}\, d\vphi = (x_1^0 - x_1^-) a_- + (x_1^+ - x_1^0) a_+,$$
and the resistance of the union of the triangles, according to \eqref{a4}, equals
$$
\int_{\Phi_4^A} d\vphi f(\nabla u(x^\vphi))\int_0^{d_\vphi} r\, dr = f(\xi_-,0) (x_1^0 - x_1^-) a_- + f(\xi_+,0) (x_1^+ - x_1^0) a_+ = b_4.
$$
Similarly to the previous case of $\Upsilon_2^A(s)$, the vertical projection of $\Upsilon_4^A(s)$ on the $x$-plane is the union of segments $\cup_{\Phi_4^A} [p_\vphi,\, q_\vphi]$. The gradient of $u(x)$ for $x$ in this projection equals $\nabla u(x^\vphi) + o(1) = (\xi_-, 0) + o(1)$ or $(\xi_+, 0) + o(1)$, when $x$ lies in the neighborhood of $I_-$ or $I_+$, respectively. Therefore  $F(\Upsilon_4^A(s))$ is the integral of $f(\nabla u(x))$ over this union of segments, that is,
$$
F(\Upsilon_4^A(s)) = \int_{\Phi_4^A} d\vphi \int_{l_\vphi}^{(1-s)l_\vphi} f(\nabla u(x^\vphi) + o(1))\, r\, dr $$
$$= [(1-s)^2 - 1] \int_{\Phi_4^A} d\vphi \int_0^{d_\vphi + O(s)} f(\nabla u(x^\vphi) + o(1))\, r\, dr $$
\beq\label{Upsilon4}
= -2s \int_{\Phi_4^A} d\vphi \int_0^{d_\vphi} \big( f(\nabla u(x)) + o(1) \big)\, r\, dr \big( 1 + o(1) \big) = -2s b_4 + o(s).
\eeq

Using formulas \eqref{predv}, \eqref{Upsilon2}, and \eqref{Upsilon4} and taking into account \eqref{aprox}, one comes to the estimate
     $$F\big(\del_0 C \setminus \big[ (1-s)C + sA \big] \big) + F\big( \del_0 C \setminus \big[ (1-s)C + sB \big] \big)
- F\big(\del_0 C \setminus \big[ (1-s)C + s(A \cup B) \big] \big) $$ $$
+ F \big( [ (1-s) \del^A C + sA ] \setminus C \big) + F \big( [ (1-s) \del^B C + sB ] \setminus C \big) = -a_2 s - 2b_4 s + o(s),
$$
and so, claim (a) in Proposition \ref{propo5} is proved.
\vspace{2mm}

Let a real number $c$ belong to the domain of the function $w$ (recall that $w(x_1) = \inf_{x_2} u(x_1,x_2)$). The intersection of the  vertical plane $x_1 = c$ with $\pl C$ is a convex curve. The lower part of this curve is the segment (maybe degenerating to a point) $x_1 = c$, $u(c,x_2) = w(c)$ parallel to the $x_2$-axis. Let it be denoted as $[A(c),\, B(c)]$, so as it is co-directional with $AB$. Note that if $c \in (x_1^-,\, x_1^+)$, then $[A(c),\, B(c)]$ belongs to $\pl_3 C$.

Now take $s<0$ and define the points $A_c(s)$ and $B_c(s)$ as follows. If the length of the segment $[A(c),\, B(c)]$ is greater than or equal to $\frac{-2\del s}{1-s}$ then set $B_c(s) = B(c)$ and take the point $A_c(s)$ on this segment so as $|A_c(s) B_c(s)| = \frac{-2\del s}{1-s}$.

If, otherwise, $|A(c) B(c)| < \frac{-2\del s}{1-s}$ then choose two points $A_c(s)$ and $B_c(s)$ on the curve $\{ x_1 = c \} \cap \pl C$ such that the segment $[A_c(s),\, B_c(s)]$ is parallel to and co-oriented with $[A,\, B]$ and its length is $\frac{-2\del s}{1-s}$. Such a choice is unique.

Consider the open arc $(\stackrel{\frown}{A_c(s), B_c(s)})$, that is, the part of the curve $\{ x_1 = c \} \cap \pl C$ which is bounded above by the points $A_c(s)$ and $B_c(s)$ and does not include these points. Denote by $X(s)$ the union of the arcs over $c \in [x_1^-,\, x_1^+]$, by $X^A(s)$ the union of arcs $(\stackrel{\frown}{A_c(s), B(c)}]$, and by $X^B(s)$ the union of arcs $(\stackrel{\frown}{B(c), B_c(s)})$ (which may be empty). That is,
$$
X(s) = \cup_{c \in [x_1^-,\, x_1^+]} (\stackrel{\frown}{A_c(s), B_c(s)}),
$$ $$
X^A(s) = \cup_{c \in [x_1^-,\, x_1^+]} (\stackrel{\frown}{A_c(s), B(c)}], \qquad X^B(s) = \cup_{c \in [x_1^-,\, x_1^+]} (\stackrel{\frown}{B(c), B_c(s)}).
$$

Below we will need the following lemma.

\begin{lemma}\label{l X}
(a) If $c \in (x_1^-,\, x_1^+)$ then $\del^A C$ contains an arc of the curve $\{ x_1 = c \} \cap \pl C$ bounded by $B(c)$ and containing $A(c)$ in its interior, and $\del^B C$ contains an open arc of this curve bounded by $B(c)$ and disjoint with $[A(c),\, B(c)]$.

(b) If $c \not\in (x_1^-,\, x_1^+)$ then an arc of the curve $\{ x_1 = c \} \cap \pl C$ containing $[A(c),\, B(c)]$ in its interior is disjoint with $\del^A C \cup \del^B C$.
\end{lemma}

\begin{proof}
(a) The plane of support at each point of $[A(c),\, B(c)]$ is parallel to $I$, and its image under the map $\pi : (x^1, x^2, z) \mapsto (x^1, z)$ is a straight line, which does not contain the point $(x^0, z^0)$ and separates this point and epi$(w)$. The plane through $I$ parallel to the original one is supporting to $\tilde C$ and does not intersect $C$. It follows that the segment $[A(c),\, B(c)]$ is contained in $\pl_3 C \subset \del^A C$.

Take an open arc of the curve $\{ x_1 = c \} \cap \pl C$ with an endpoint at $A(c)$ disjoint with $[A(c),\, B(c)]$. Draw the tangent line to a point on this arc in the vertical plane $\{ x_1 = c \}$ and consider the line through $A$ parallel to it. The point $B$ and epi$(w)$ are situated above this line. If the arc is chosen sufficiently small then the plane through $A$ parallel to the plane of support at a point of this arc does not intersect $C$ and is situated below $B$. It follows that the chosen arc belongs to $\pl_1^A C \subset \del^A C$. A similar argument holds for $\del^B C$.

(b) In this case the image of the plane of support at a point of $[A(c),\, B(c)]$ under the map $\pi$ is a straight line, the point $(x^0, z^0)$ and the set epi$(w)$ are situated on one side of this line, and the line does not contain $(x^0, z^0)$. It follows that $[A(c),\, B(c)]$ belongs to $\pl_0 C$, and therefore, is disjoint with $\del^A C \cup \del^B C$, and the same is true for an open arc of the curve $\{ x_1 = c \} \cap \pl C$ containing $[A(c),\, B(c)]$.
\end{proof}

It follows from claim (a) of Lemma \ref{l X} that
 \beq\label{os1}
\text{the 2-dimensional Lebesgue measures of} \ \ X^A(s) \setminus \del^A C \ \ \text{and} \ \ X^B(s) \setminus \del^B C  \ \ \text{are} \ \ o(s).
\eeq
 From claim (b) it follows that
 \beq\label{os2}
\text{the measures of} \ \ \cup_{c \not\in (x_1^-,\, x_1^+)} (\stackrel{\frown}{A_c(s), B(c)}] \cap \del^A C \ \
\text{and} \ \ \cup_{c \not\in (x_1^-,\, x_1^+)} (\stackrel{\frown}{B(c), B_c(s)}) \cap \del^B C \ \ \text{are} \ \ o(s).
\eeq

The orthogonal projection of $X(s)$ on the $x$-plane is contained in the strip $x_1^- \le x_1 \le x_1^+$, and its intersection with each line $x_1 = c$ is a segment with the length $\frac{-2\del s}{1-s}$. That is, its length and width are $x_1^+ - x_1^-$ and $\frac{-2\del s}{1-s}$. The gradient of $u$ at a point $(x_1, x_2)$ in this projection is $(w'(x_1), 0) + o(1)$. Therefore, by \eqref{a3},
\beq\label{omalo}
F(X(s)) = \frac{-2\del s}{1-s} \int_{x_1^-}^{x_1^+} f\big((w'(x_1), 0) + o(1)\big)\, dx = -a_3 s (1 + o(1)).
\eeq

Now consider the sets
$$
[(1-s) \del^A C + sA] \setminus [ (1-s)C + sB ] \quad \text{and}  \quad [(1-s) \del^B C + sB] \setminus [(1-s) C + sA].
$$
The homothety with the center at $A$ and ratio $\frac{1}{1-s}$ of the former set and the homothety with the center at $B$ and the same ratio of the latter   set are, respectively,
\beq\label{2sets}
\del^A C \setminus \big[ C + s' (B - A) \big] \quad \text{and}  \quad \del^B C \setminus \big[ C + s' (A - B) \big], \quad \text{where} \ \, s' = \frac{s}{1-s}.
\eeq
These sets are disjoint, and
$$
F\big([(1-s) \del^A C + sA] \setminus [ (1-s)C + sB ] \big) + F\big( [(1-s) \del^B C + sB] \setminus [(1-s) C + sA] \big)
$$
\beq\label{res2}
= (1-s)^2 F \Big( \big( \del^A C \setminus \big[ C + s' (B - A) \big] \big) \cup \big( \del^B C \setminus \big[ C + s' (A - B) \big] \big) \Big).
\eeq

Take a point $P \in \del^A C$. The intersection of the ray $P + \lam(B - A),\, \lam \ge 0$ with $C$ is a closed segment $PQ$. The point $P$ lies in $\del^A C \setminus \big[ C + s' (B - A) \big]$ if and only if $|PQ| < -s' |I| = \frac{-2\del s}{1-s}$.

Let now $Q \in \del^B C$. The intersection of the ray $Q - \lam(B - A),\, \lam \ge 0$ with $C$ is a closed segment $PQ$. The point $Q$ lies in $\del^B C \setminus \big[ C + s' (A - B) \big]$ if and only if $|PQ| < -s' |I|$.

We conclude that the intersection of the set $\del^A C \setminus \big[ C + s' (B - A) \big]$ with the layer $\{ x_1 \in (x_1^-,\, x_1^+) \}$ coincides with $X^A(s) \cap \del^A C$, and the intersection of $\del^B C \setminus \big[ C + s' (A - B) \big]$ with this layer coincides with $X^B(s) \cap \del^B C$.
On the other hand, the intersections of $\del^A C \setminus \big[ C + s' (B - A) \big]$ and $\del^B C \setminus \big[ C + s' (A - B) \big]$ with the set $\{ x_1 \not\in (x_1^-,\, x_1^+) \}$ coincide, respectively, with $\cup_{c \not\in (x_1^-,\, x_1^+)} (\stackrel{\frown}{A_c(s), B(c)}] \cap \del^A C$ and $\cup_{c \not\in (x_1^-,\, x_1^+)} (\stackrel{\frown}{B(c), B_c(s)}) \cap \del^B C$.

It follows from \eqref{os1} and \eqref{os2} that the symmetric difference of the sets $\big( \del^A C \setminus \big[ C + s' (B - A) \big] \big) \cup \big( \del^B C \setminus \big[ C + s' (A - B) \big] \big)$ and $X(s)$ has  the 2-dimensional measure $o(s)$. Using \eqref{omalo}, we then obtain
$$
F \Big( \big( \del^A C \setminus \big[ C + s' (B - A) \big] \big) \cup \big( \del^B C \setminus \big[ C + s' (A - B) \big] \big) \Big) = -a_3 s (1 + o(1)),
$$
and so, by \eqref{res2},
$$
F\big([(1-s) \del^A C + sA] \setminus [ (1-s)C + sB ] \big) + F\big( [(1-s) \del^B C + sB] \setminus [(1-s) C + sA] \big) = -a_3 s (1 + o(1)).
$$
Claim (b) in Proposition \ref{propo5} is proved.
\vspace{2mm}

It remains to show that the resistances of the sets $\SSS_1 = \big[ (1-s) \del^A C + sA \big] \setminus \big( C \cup \big[ (1-s)C + sB \big] \big)$ and
$\SSS_2 = \big[ (1-s) \del^B C + sB \big] \setminus \big( C \cup \big[ (1-s)C + sA \big] \big)$ are $o(s)$. We will consider only the set $\SSS_1$; the proof for $\SSS_2$ is the same.

Consider a point $P \in \del^A C$ and draw the ray $AP$ with the vertex at $A$. The intersection of this ray with $C$ is a non-degenerate segment $[P,\, P']$. The point $P$ belongs to $\frac{1}{1-s} C + \frac{-s}{1-s} A$ if and only if $|PP'| \ge -s|AP|$. The ratio $\frac{|PP'|}{|AP|}$ is a positive continuous function of $P$; it follows that any compact subset of $\del^A C$ belongs to $\frac{1}{1-s} C + \frac{-s}{1-s} A$ for $|s|$ sufficiently small. In particular, this is true for the compact set $\overline{X^A(h(\kappa))} \cap \{ x_1^- + \kappa \le x_1 \le x_1^+ - \kappa \}$, where the negative function $h(\kappa)$ going to 0 as $\kappa \to 0$ is chosen so as the above set belongs to $\del^A C$.

The homothety with the center $A$ and the ratio $\frac{1}{1-s}$ of $\SSS_1$ is the set-theoretic difference of the sets $\del^A C \setminus \big[ C + \frac{s}{1-s} (B - A) \big]$ and $\frac{1}{1-s} C + \frac{-s}{1-s} A.$ It is proved above that the intersection of the former set with the domain $\{ x_1 < x_1^- \} \cup \{ x_1 > x_1^+ \}$ has the 2-dimensional measure $o(s)$. Further, the intersection of the former set with the layer $x_1^- \le x_1 \le x_1^+$ for $|s| \le |h(\kappa)|$ belongs to $\overline{X^A(h(\kappa))} \cap \{ x_1^- + \kappa \le x_1 \le x_1^+ - \kappa \}$, and therefore, belongs to $\frac{1}{1-s} C + \frac{-s}{1-s} A$. Finally, the intersection of the aforementioned set-theoretic difference with $\{ x_1 \in [x_1^-,\, x_1^- + \kappa] \cup [x_1^+ - \kappa,\, x_1^+] \}$ has the 2-dimensional measure $\le 2\kappa \cdot 2\del |s| (1 + o(1))$. Since $\kappa$ can be chosen arbitrarily small, the measure of the set-theoretic difference is $o(s)$.
Claim (c) in Proposition \ref{propo5} is proved.

\section*{Acknowledgements}

This work is supported by CIDMA through FCT (Funda\c{c}\~ao para a Ci\^encia e a Tecnologia), reference UIDB/04106/2020.
I am very grateful to A. I. Nazarov for fruitful discussions.

\end{document}